\documentclass[11pt]{amsart}

\usepackage{comment}
\usepackage{geometry}
\usepackage{graphicx}
\usepackage{mathtools}
\usepackage{amsthm, amsmath, amssymb}
\usepackage{yfonts}
\usepackage{algorithm}
\usepackage{algpseudocode}
\usepackage[dvipsnames]{xcolor}
\usepackage{nicematrix}
\usepackage{multicol}
\usepackage{hyperref}
\usepackage{cleveref}
\usepackage{thmtools} 
\usepackage[textsize=small]{todonotes} 

\usepackage{subcaption}  
\usepackage{caption}     
\usepackage[normalem]{ulem} 

\newcommand{\C}{\mathbb{C}}     
\newcommand{\R}{\mathbb{R}}     
\newcommand{\N}{\mathbb{N}}     
\newcommand{\ambSp}[1]{\mathcal{F}_{#1}} 
\newcommand{\xx}{\textbf{x}}    
\newcommand{\zz}{\textbf{z}}    
\newcommand{\dd}{\mathbf{d}}    
\newcommand{\ww}{\mathbf{w}}    
\newcommand{\e}{\mathbf{e}}

\DeclareMathOperator{\Sym}{Sym}
\DeclarePairedDelimiter\floor{\lfloor}{\rfloor} 

\usepackage[T1]{fontenc}

\newtheorem{theorem}{Theorem}[section]
\newtheorem{proposition}[theorem]{Proposition}
\newtheorem{lemma}[theorem]{Lemma}
\newtheorem{conjecture}[theorem]{Conjecture}

\newtheorem{problem}[theorem]{Problem}

\newtheoremstyle{nonitalic}
  {3pt} 
  {3pt} 
  {} 
  {} 
  {\bfseries} 
  {.} 
  { } 
  {} 
\theoremstyle{nonitalic}

\newtheorem{example}[theorem]{Example}
\newtheorem{remark}[theorem]{Remark}
\newtheorem{definition}[theorem]{Definition}

\begin{document}

\title{Algebraic geometry of rational neural networks}
\author[Alexandros Grosdos]{Alexandros Grosdos}
\address{Institute of Mathematics, University of Augsburg, Universitätsstrasse 14, 86159 Augsburg}
\email{alexandros.grosdos@uni-a.de}
\author{Elina Robeva}
\address{Department of Mathematics, The University of British Columbia, 1984 Mathematics Road, Vancouver, BC, Canada, V6T 1Z2}
\email{erobeva@math.ubc.ca}
\author{Maksym Zubkov}
\address{Department of Mathematics, The University of British Columbia, 1984 Mathematics Road, Vancouver, BC, Canada, V6T 1Z2}
\email{mzubkov@math.ubc.ca}


\date{\today}

\begin{abstract}
    We study the expressivity of rational neural networks (RationalNets) through the lens of algebraic geometry. 
    We consider rational functions that arise from a given RationalNet to be tuples of fractions of homogeneous polynomials of fixed degrees. 
    For a given architecture, the neuromanifold is the set of all such expressible tuples. 
    For RationalNets  with one hidden layer 
    and fixed activation function $\sigma(x)=1/x$,
    we characterize the dimension of the neuromanifold and provide defining equations for some architectures. 
    We also propose algorithms that determine whether a given rational function belongs to the neuromanifold. For deep binary RationalNets, i.e., RationalNets all of whose layers except potentially for the last one are binary, we classify when the Zarisky closure of the neuromanifold equals the whole ambient space, and give bounds on its dimensions.

\end{abstract}
\maketitle

\section{Introduction}

Neural networks are parameterized families of functions
with an enormous range of applications in the sciences
due to their versatility as universal function approximators.
When the activation function is polynomial, the corresponding function space, 
also known as the neuromanifold, 
can be studied algebraically
as it can be naturally described by polynomial equations and inequalities. 
Recent work in this area \cite{KUBJAS24} has provided remarkable insight into the geometry of these networks, 
allowing us to better understand 
the class of representable functions,
as well as the loss landscape.
More broadly, a conceptual dictionary relating the algebro-geometric invariants of varieties to fundamental notions in the theory of neural netowrks has been proposed in~\cite{MARCHETTI25}.

Due to the inherent limitations of polynomials
in approximation theory \cite{AUSTIN21}, rational activation functions have
become widely used by both theorists and practitioners
\cite{BOULLE22,ALEJANDRO19}, and
RationalNets have already shown superior performance in function approximation \cite{BOULLE20}.

In this work, we present the first algebro-geometric study of the neuromanifold and its Zariski closure, called the {\em neurovariety}, associated with a given rational neural network.

We use computational algebraic geometry
and tensor decompositions 
to understand the space of representable functions (the neuromanifold) by a given network architecture
and to reconstruct the corresponding parameters of the neural network.

Our smallest motivating example involves a neural network with $2$ layers as follows.

\begin{example}
\label{example:motivaitonal-example}
Consider the 2-layer neural network pictured below. Given weight
     matrices $W_1 = (a_{ij}) \in \R^{2\times 2}$, $W_2 = (b_{ij}) \in \R^{1\times 2}$ with nonzero rows, 
    define the corresponding linear functions $\alpha_1$ and $\alpha_2$,
    and let $\sigma$ be the rational activation function that acts entrywise as $x \mapsto 1/x$.
    This $2$-layer rational neural network then gives rise to all functions 
    $\R^2 \to \R$ of the form
    \[f_{\ww}(\xx) = (\alpha_2 \circ \sigma \circ \alpha_1)(\xx)=
    \frac{(b_1a_{21}+b_2a_{11})x_1+(b_1a_{22}+b_2a_{12})x_2}{(a_{11}x_1+a_{12}x_2)(a_{21}x_1+a_{22}x_2)}.
    \]
    The function space consists of all functions of the above form, 
    namely all functions that can be written as a quotient 
    \begin{equation}
    \label{eq:221-architecture-rat-funs-for}
        \frac{C_{10}x_1+C_{01}x_2}{C_{20}x_1^2+C_{11}x_1x_2+C_{02}x_2^2} \in \mathbb{R}(x_1,x_2)
    \end{equation}
    with the restriction that the denominator factors as a product of two real linear forms. In other words, the inequality $C_{11}^2-4C_{20}C_{02} \geq 0$ holds.
\end{example}

\subsection{Previous Work}

A comprehensive overview of the notion of expressivity of neural networks can be found in~\cite{GUHRING20}.

The most common approach to training a neural network is stochastic gradient descent with respect to a given \textit{loss function} \cite{BOTTOU10}. The goal is to minimize this loss function over the space of parameters $\ww\in\R^N$. Alternatively, the training process can be viewed as an optimization over the function space~\cite{MARCHETTI25}. 

One of the earliest works~\cite{AMARI94} to study the geometric properties of the function space of a neural network also proposed the term {\em neuromanifold} to denote the function space of a given neural network architecture.

Studying the training process as an optimization problem over the neuromanifold has seen recent progress from the applied algebra and geometry community in cases when the activation function is either polynomial or ReLU. The role of depth in the expressive power of ReLU neural networks was studied in~\cite{HANIN19}. The connection between ReLU neural networks and tropical geometry was first established in~\cite{ZHANG18}. The expressivity of PNNs was studied in \cite{KILEEL19} through the lens of filling and thick architectures. The geometry of neuromanifolds and neurovarieties was explored in \cite{KUBJAS24}, and the expressive power of PNNs  was further examined in~\cite{FINKEL25}. The identifiability of PNNs was studied in~\cite{USEVICH25}, while the comprehensive study of singularities of PNNs was provided in~\cite{SHAHVERDI25-1}. For polynomial convolutional neural networks, the expressive power of architectures and questions related to geometry and optimization have been investigated in~\cite{SHAHVERDI25-2}.

One drawback of using rational activation functions is the possibility of encountering singularities, as the denominator may become zero. One of the ways to solve this issue is to use Padé approximation units \cite{BREZINSKI96, ALEJANDRO19}, which help keep the denominator away from zero and make RationalNets a universal approximator. In classical rational function approximation theory, rational approximations have shown significantly better performance than polynomial approximations~\cite{BORWEIN83, NEWMAN78}.

Having a discontinuity in the activation function is not uncommon. For instance, the JumpReLU activation function which contains a jump discontinuity has been trained and studied in \cite{ERICHSON19} where it is shown to increase the model robustness against adversarial attacks. 

To achieve more flexibility and accuracy during training, the coefficients of the rational activation functions for each layer can also be treated as trainable parameters alongside the network weights~\cite{BOULLE20}. This also allows for the use of significantly fewer parameters in RationalNets than in neural networks with ReLU activation functions. Such networks have been successfully applied to learn Green's functions of linear partial differential equations using physically informed neural networks \cite{BOULLE22}.

An alternative approach to studying the expressivity of RationalNets is through Taylor varieties \cite{CONCA23}, where one considers the formal Taylor expansion of a rational function. General rational approximation of multivariate functions is discussed in \cite{AUSTIN21}. We leave this approach to future work.

Lastly, the function space of shallow single-output RationalNets with activation function $\sigma(x) = 1/x$ is related to the space of meromorphic functions in several variables with linear poles \cite{DAHMEN24} which arises in quantum field theory from Feyman integrals \cite{CLAVIER19}, in number theory from multi zeta functions \cite{MANCHON10}, and algebraic geometry from Jeffrey-Kirwan residue \cite{JEFFREY97}.

\subsection{Structure of the Paper and Main Contributions}

The rest of the paper is organized as follows.

In Section \ref{section:preliminaries}, we define rational neural networks and show that they can be written as vectors of polynomial ratios.
The degrees of these polynomials are then  computed recursively.
We identify the neurovariety, 
i.e., the Zariski closure of the set of representable functions,
with a variety in the ambient space of tuples of polynomials of fixed degree.

In Section~\ref{section:shallow-networks} we study shallow neural networks from the point of view of tensor decompositions.
Algebraic techniques allow us to decompose the symmetric tensors induced by the denominators of elements in the function space of the neural network using two different methods.

In Section \ref{sec:ideals} we study shallow neural networks. We give a full description of filling architectures, characterize the neurovarieties for architectures with two neurons in the middle layer, 
and provide a partial characterization when the middle layer is larger.

In Section \ref{section:binaryNN} we discuss deep binary rational neural networks,
whose expressivity is particularly appealing for applications.
We provide a description of the neural network as a function of its depth and study algebraic aspects, 
like its dimension and whether the neuromanifold/neurovariety fills the whole ambient space or not. 

In Section~\ref{sec:numerics}, we numerically compute the dimension of the neurovariety and provide an example of a rational neural network with activation function $\sigma(x) = 1/x$ learning meromorphic functions from data. 

Finally, in Section~\ref{sec:conclusion}, we provide the summary of the paper and some possible future directions.

The source code used in this paper is available at \url{https://github.com/maxzubkov/rationalnets}

\section{Preliminaries}
\label{section:preliminaries}

In this section, we collect the necessary background material on neural networks and introduce RationalNets along with their corresponding parameter map. 
In Section~\ref{subsec:2-neural-networks-and-neeuromanifolds}, we present a general definition of a feedforward neural network and discuss
the ambient space, the parameter map, and the neuromanifold. 
In Section~\ref{subsec:2-shape-of-rational-neural-networks}, we define RationalNets with activation function  $1/x$ and describe their closed-form expression. In Section~\ref{subsec:2-combinatorial-param-map}, we construct a {``combinatorial''} parameter map, introduce the associated neurovariety, and list the architectures and main questions that we aim to study.

\subsection{Neural networks and neuromanifolds}
\label{subsec:2-neural-networks-and-neeuromanifolds}
Let $\dd=(d_0,d_1,\dots,d_L)$ be an $L$-tuple of natural numbers. A \textit{feedforward neural network} $f_{\ww}:\R^{d_0}\to\R^{d_L}$ is a composition of affine-linear maps $\alpha_{i}:\R^{d_{i-1}}\to\R^{d_i}$ and non-linear maps $\sigma_i:\R^{d_i}\to\R^{d_i}$ given by
\[
    f_{\ww}(\xx)\coloneqq (\alpha_L\circ\sigma_{L-1}\circ \alpha_{L-1}\circ\dots\circ\sigma_1\circ \alpha_1)(\xx).
\]
The vector $\ww=(W_1\dots,W_L,b_1,\dots,b_L)$ is the \textit{parameter vector} of the neural network, where $W_i\in\R^{d_{i-1}\times d_i}$ and $b_i\in\R^{d_{i}}$, and the affine linear maps $\alpha_i$ are given by $\alpha_i(\xx)=W_i\xx+b_i$. The matrices $W_1,\dots,W_L$ and the vectors $b_1,\dots,b_L$ are called \textit{weights} and \textit{biases} of the neural network, respectively. 

The map $\sigma_i:\R^{d_i}\to\R^{d_i}$ is called an \textit{activation} map and is here the coordinate-wise application of a \textit{activation function} $\sigma:\R\to\R$. The \textit{architecture} of the neural network is the pair $(\dd,\sigma)$.

Depending on the choice of activation function $\sigma$, the image of the neural network $f_{\ww}$ belongs to a different \textit{ambient space}. 
For example, if $\sigma:\R\to\R$ is any continuous function, 
then the neural network $f_{\ww}$ belongs to the space of continuous functions from $\R^{d_0}$ to $\R^{d_L}$. 
Let $\ambSp{}(\R^{d_0},\R^{d_L})$ be the space of all functions from $\R^{d_0}$ to $\R^{d_L}$, and let $\R^N$ be the space of all neural network parameters, where $N$ is the total number of all weights and biases. For an arbitrary choice of $\sigma$, we have that $f_{\ww}\in \ambSp{}(\R^{d_0},\R^{d_L})$.  

We can further define the \emph{parameter map}, which we denote by $\Psi_{\dd, \sigma}$, that takes an element $\ww\in\R^N$ and maps it to an element in the ambient space $\ambSp{}(\R^{d_0}, \R^{d_L})$ as follows:
\begin{equation}
\label{eq:parameter-map-general}
    \Psi_{\dd,\sigma}:\R^{N}\to \ambSp{}(\R^{d_0},\R^{d_L}),\hspace{0.2cm} \ww\mapsto f_{\ww}.    
\end{equation}
For different choices of parameters $\ww$, we obtain different points $f_{\ww}$ in the image of the map $\Psi_{\dd,\sigma}$. So, all possible functions $f_{\ww}$ that a fixed neural network architecture can express within the ambient space $\ambSp{}(\R^{d_0},\R^{d_L})$ are the image of the parameter map $\Psi_{\dd,\sigma}$. 

\begin{definition}
    The image of the map $\Psi_{\dd,\sigma}$ is called the \textit{neuromanifold} $\mathcal{M}_{\dd,\sigma}$.
\end{definition}

In this work, when we talk about the \textit{expressive power},
we follow the notion given in \cite{KILEEL19} where it is defined as the ability of the neural network to exactly learn a given function.

 Depending on the choice of activation function $\sigma$, the neural network $f_{\ww}$ can have very specific properties. This allows us to shrink  the ambient space $\ambSp{}(\mathbb R^{d_0}, \mathbb R^{d_L})$ to a smaller subspace $\mathcal F_\sigma(\mathbb R^{d_0}, \mathbb R^{d_L})$ and study the neuromanifold $\mathcal M_{\dd, \sigma}$ within this new subspace. For example, if 
 the activation function is $\sigma(x)=\text{ReLU}(x)$, then we can pick  $\mathcal F_\sigma(\mathbb R^{d_0}, \mathbb R^{d_L})$ as our ambient space, namely
the space of piecewise linear functions from $\mathbb R^{d_0}$ to $\mathbb R^{d_L}$ \cite{HANIN19}.

In the case of a polynomial activation function $\sigma$, the ambient space $\mathcal F_{\sigma}(\R^{d_0},\R^{d_L})$ can be chosen to be finite-dimensional. Indeed, the output function $f_\ww$ is always a tuple of polynomials of given bounded degree, and the space of such polynomials is a finite-dimensional vector space \cite{KILEEL19}. In comparison, any {continuous} non-polynomial activation function yields an infinite-dimensional ambient space $\mathcal F_\sigma(\mathbb R^{d_0}, \mathbb R^{d_L})$ despite the fact that the neuromanifold $\mathcal{M}_{\dd,\sigma}$ is a finite-dimensional embedding of the parameter space $\R^N$ in $\mathcal{F}_{\sigma}(\R^{d_0},\R^{d_L})$ \cite{LESHNO93}. This infinite-dimensional setting limits the applicability of tools from algebraic geometry, which traditionally require embeddings into finite-dimensional projective spaces. 

This issue arises in particular when we have a rational activation function $\sigma$. In this case, the output function $f_{\ww}$ is a $d_L$-tuple of rational functions, and, therefore, it lies in the  infinite-dimensional space of rational functions.

In order to work in a finite-dimensional setting, we represent the output of a
fixed rational neural network architecture as a tuple of rational functions whose numerators and denominators have bounded degrees.
We then treat each rational function $P/Q$ as a point $(P,Q)$ in a space parametrized by the coefficients of the numerator $P$ and the denominator $Q$. However, note that the map
\[
    P/Q\mapsto(P,Q) 
\]
is not well-defined, since distinct pairs can represent the same rational function. For example, $\frac{x}{xy}$ and $\frac{y}{y^2}$ both simplify to the same rational function $\frac{1}{y}$. One way to resove this issue is to remove the points where the resultant $Res(P,Q)$ vanishes, as discussed in \cite{SILVERMAN07}. In this work however, each output function has uniquely defined numerator and denominator arising from the neural network parametrization, allowing us to easily avoid this issue.

In the next subsection, we study the explicit form of the numerator and denominator polynomials when the activation function is  $\sigma(x)=1/x$.

\subsection{Rational neural networks}
\label{subsec:2-shape-of-rational-neural-networks}

A \textit{rational neural network (RationalNet)}  $f_{\ww}:\R^{d_0}\to\R^{d_L}$, with architecture $(\dd,\sigma)$, is the composition of functions 
\[
    \R^{d_0}\xrightarrow{W_1} \R^{d_1}\xrightarrow{\sigma}\R^{d_1}\xrightarrow{W_2} \dots \xrightarrow{\sigma} \R^{d_{L-1}} \xrightarrow{W_{L}} \R^{d_L},
\]
where $W_{k}\in \R^{d_k\times d_{k-1}}$ are linear maps and $\sigma$ acts coordinate-wise as $x\mapsto1/x$.
We use $W_k$ for both the matrix and the linear function it induces. We denote the $(i,j)$-th entry of the matrix $W_k$ by $w_{kij}$.
More precisely, the output of the network is the function
\begin{equation}
\label{eq:rnn-main-equation}
    f_{\ww}(\xx) = (W_L\circ\sigma\circ W_{L-1}\circ\dots\circ\sigma\circ W_1)(\xx).
\end{equation}

In our study of neural networks we think of the entries of the matrices $W_k$ as variables in a polynomial ring. Since we are primarily interested in the algegbraic and geometric properties of the neurovariety, we avoid a discussion  of the domain where each function $f_{\ww}$ is defined.

We denote the $i$th component of $f_{\ww}$ by $f_{i,\ww}$ and observe that it is a rational function of the form $P_{i,\ww} / Q_{i,\ww}$. In fact, we show that all denominators $ Q_{i,\ww}$ are equal, and we denote them by $Q_{\ww}$.

Let $S^{d}(\R^{n})$ be the space of homogeneous polynomials of degree $d$ over $n$ variables with coefficients in $\R$, and let $n(\dd)$ and $m(\dd)$ be the degrees of the numerators $P_{i,\ww}$ and the denominator $Q_{\ww}$, respectively.  Then, for the numerators we have $P_{i,\ww} \in S^{n(\dd)}(\R^{d_0})$ and for the denominator we have $Q_{\ww} \in S^{m(\dd)}(\R^{d_0})$ for all $i = 1, \ldots, d_L$, and we compute the degress $n(\dd)$ and $m(\dd)$ in Lemma~\ref{lemma:total-degree}. 

In the remainder of this subsection, we express the neural network via a recursive formula and study the symmetries of the fibers of the map $\Psi_{\dd,\sigma}(\ww)=f_{\ww}$.

\begin{theorem}
\label{theorem:neural-network-closed-form-general}
    Consider the neural network $f_{\ww}:\R^{d_0}\to\R^{d_L}$ with dimensions $\dd=(d_0,d_1,\dots,d_L)$ where $d_i\geq2$ for $i=0,\dots,L-1$, weight matrices $W_k\in\mathbb{R}^{d_k\times d_{k-1}}$, and activation function $\sigma(x)$ acting entrywise $x \mapsto 1/x$. 
    Then the neural network output $f_{\ww}$ is a $d_L$-tuple of rational functions
    \[
        f_{\ww}(\xx)=\left(\frac{P_{1,\ww}(\xx)}{Q_{\ww}(\xx)},\dots,\frac{P_{d_L,\ww}(\xx)}{Q_{\ww}(\xx)}\right)^\top,
    \]
    where $P_{i,\ww}$ and $Q_{\ww}$ are homogeneous polynomials which factorize in the form
    \begin{equation}
    \label{eq:RNN-coordinates-formulas}
    \begin{aligned}
        P_{i,\ww}(\xx) &= p^{(L)}_{i}(\xx)q^{(L-1)}(x)q^{(L-3)}(\xx)\dots q^{(\delta(L+1))}(\xx),\\
        Q_{\ww}(\xx) &= q^{(L)}(\xx)q^{(L-2)}(\xx)\dots q^{(\delta(L))}(\xx),
    \end{aligned}
    \end{equation}
    where $\delta(\cdot)$ denotes the parity ($\delta(L)=0$ if $L$ is even, $\delta(L)=1$ if $L$ is odd) and the polynomials $p^{(k)}_{i}$ and $q^{(k)}$ are computed recursively as follows.
    
    \begin{itemize}
        
    \item We initialize:
    \[
        p^{(1)}_{i}(\xx)\coloneqq\sum_{j=1}^{d_0}w_{1,i,j}x_j,\quad q^{(k)}(\xx)\coloneqq 1 \text{ for }k=0,1.
    \]
    
    \item For $k\geq 1$, we define:
    \begin{equation}
    \label{eq:main-recursion-formula}
        p^{(k+1)}_{i}(\xx)\coloneqq\sum_{j=1}^{d_k}w_{k+1,i,j}\prod_{\substack{s=1\\ s\neq j}}^{d_k}p^{(k)}_{s}(\xx),\quad
        q^{(k+1)}(\xx)\coloneqq\prod_{j=1}^{d_k} p^{(k)}_{j}(\xx).  
    \end{equation}
    
    \end{itemize}
\end{theorem}
\begin{proof}
    See Appendix~\ref{apx:neural-network-closed-form-general}
\end{proof}

We require all hidden layers and the input of $f_{\ww}$ to have dimension at least two. This is because if layer $k$ has dimension one, i.e., the network architecture contains a segment $(d_{k-1},1,d_{k+1})$, then the degrees $n(\dd)$ and $m(\dd)$ stop growing after layer $k$. 

Indeed, consider the architecture $\dd=(1,d_1,d_2)$, then the network output is equal to
\[
    t\xrightarrow{W_1} \begin{bmatrix}
        w_{111}t\\
        \vdots\\
        w_{1d_11}t
    \end{bmatrix}\xrightarrow{\sigma}\begin{bmatrix}
        1/(w_{111}t)\\
        \vdots\\
        1/(w_{1d_11}t)
    \end{bmatrix}\xrightarrow{W_2}\begin{bmatrix}
        \sum_{j=1}^{d_1}\frac{w_{21j}}{w_{1j1}t}\\
        \vdots\\
        \sum_{j=1}^{d_1}\frac{w_{2d_2j}}{w_{1j1}t}
    \end{bmatrix}=\begin{bmatrix}
        \frac{\sum_{j=1}^{d_1}w_{21j}\prod_{s\neq j,s=1}^{d_1}w_{1s1}}{t\prod_{s=1}^{d_1}w_{1s1}}\\
        \vdots\\
        \frac{\sum_{j=1}^{d_1}w_{2d_2j}\prod_{s\neq j,s=1}^{d_1}w_{1s1}}{t\prod_{s=1}^{d_1}w_{1s1}}
    \end{bmatrix}
\]
After canceling the common factor $t$, neither the numerator nor the denominator increase in degree. Therefore, any hidden width of dimension $1$ stops the growth of $n(\dd)$ and $m(\dd)$.

To illustrate the recursion formula~\eqref{eq:main-recursion-formula}, we take a look at the output of the architecture $\dd=(3,3,1)$.
\begin{example}
    Let $p_{i}^{(1)}$ be the $i$th component of the vector $W_1\xx$ and set $q^{(1)}\coloneqq1$. Then we have
    \[
    \begin{bmatrix}
        x_1\\
        x_2\\
        x_3
    \end{bmatrix}\xrightarrow{W_1} \begin{bmatrix}
        p_{1}^{(1)}/{q^{(1)}}\\
        p_{2}^{(1)}/{q^{(1)}}\\
        p_{3}^{(1)}/{q^{(1)}}
    \end{bmatrix}\xrightarrow{\sigma}\begin{bmatrix}
        q^{(1)}/{p_{1}^{(1)}}\\
        q^{(1)}/{p_{2}^{(1)}}\\
        q^{(1)}/{p_{3}^{(1)}}
    \end{bmatrix}\xrightarrow{W_2}
    \frac{q^{(1)}}{{p_{1}^{(1)}p_{2}^{(1)}p_{3}^{(1)}}}\begin{bmatrix}
        w_{111}p_{2}^{(1)}p_{3}^{(1)}+w_{112}p_{1}^{(1)}p_{3}^{(1)}+w_{113}p_{1}^{(1)}p_{2}^{(1)}\\
        w_{121}p_{2}^{(1)}p_{3}^{(1)}+w_{122}p_{1}^{(1)}p_{3}^{(1)}+w_{123}p_{1}^{(1)}p_{2}^{(1)}\\
       w_{131}p_{2}^{(1)}p_{3}^{(1)}+w_{132}p_{1}^{(1)}p_{3}^{(1)}+w_{133}p_{1}^{(1)}p_{2}^{(1)}
    \end{bmatrix}.
    \]
    Hence, we can set 
    \[
        p^{(2)}_{i}=w_{1i1}p_{2}^{(1)}p_{3}^{(1)}+w_{1i2}p_{1}^{(1)}p_{3}^{(1)}+w_{1i3}p_{1}^{(1)}p_{2}^{(1)},\quad q^{(2)}={p_{1}^{(1)}p_{2}^{(1)}p_{3}^{(1)}}
    \]
    and iterate the same procedure for the deeper layers.
\end{example}

In Theorem~\ref{theorem:neural-network-closed-form-general} we expressed $P_{i,\ww}$ and $Q_{\ww}$ recursively for general architectures. Providing a
closed-form formula for $P_{i,\ww}$ and $Q_{\ww}$ however is notation-heavy and cumbersome, and left as a challenge for upcoming work.
We partially address this challenge in Lemma~\ref{lemma:closed-tensor-form-shallow-network} and Proposition~\ref{prop:closed-form-binary-deep-network}, where we provide explicit closed-form expressions for shallow and binary deep rational neural networks, respectively.

Next, we compute the dimension of the ambient space by determining the degrees $n(\dd)$ and $m(\dd)$ of $P_{i, \ww}$ and $Q_{\ww}$. First, we need to compute the degrees of the polynomials $p_i^{(k)}$ and $q^{(k)}$.

\begin{lemma}
\label{lemma:degrees-for-p-and-q}
    Let $p^{(k)}_i$ and $q^{(k)}$ be defined recursively as in equation~\eqref{eq:main-recursion-formula}, where the dimensions $d_j \geq 2$ for $j=0,\dots,L-1$. Then, the degree of $p^{(k)}_i$ is
    \[
        \deg(p^{(k)}_i) = \prod_{j=1}^{k-1} (d_j - 1) = (d_1 - 1)\cdots(d_{k-2} - 1)(d_{k-1} - 1),
    \]
    and the degree of $q^{(k)}$ is
    \[
        \deg(q^{(k)}) = d_{k-1} \prod_{j=1}^{k-2} (d_j - 1) = d_{k-1}(d_1 - 1)\cdots(d_{k-3} - 1)(d_{k-2} - 1).
    \]
\end{lemma}
\begin{proof}
    See Appendix~\ref{apx:degrees-for-p-and-q}
\end{proof}

\begin{lemma}
\label{lemma:total-degree}
    Let $\dd=(d_0,d_1,\dots,d_L)$ with $d_i\geq 2$ for $i=0,1,\dots,L-1$ and $\sigma(x)=1/x$. Let $P_{i,\ww}$ and $Q_{\ww}$ be defined in \eqref{eq:RNN-coordinates-formulas}. Then the degrees of $P_{i,\ww}$ and $Q_{\ww}$ are given by
    \[
        \deg(P_{i,\ww})= \prod_{j=1}^{L-1} \left( d_{j} - 1 \right)+\sum_{k=1}^{\floor{\frac{L}{2}}+1}d_{L-2k}\prod_{j=1}^{L-2k-1} \left( d_{j} - 1 \right),
    \]
    \[
        \deg(Q_{\ww}) = \sum_{k=1}^{\left\lfloor \frac{L}{2} \right\rfloor + 1} d_{L - 2k + 1} \prod_{j=1}^{L-2k} \left( d_{j} - 1 \right).
    \]
\end{lemma}
\begin{proof}
    See Appendix~\ref{apx:total-degree}
\end{proof}

Finally, if $\R(x_1,\dots,x_n)$ denotes the space of rational functions in $n$ variables, then the next lemma computes some of the symmetries of the fibers of the parameter map $\Psi_{\dd,\sigma}:\R^N\to(\R(x_1,\dots,x_{d_0}))^{d_L}$. In the case of a monomial activation function, the symmetries were computed in~\cite{KILEEL19,KUBJAS24}.

\begin{lemma}
\label{lemma:fiber-symmetries-general}
    Let $\dd = (d_0, d_1, \dots, d_L)$ and $\sigma$ be the entrywise $x \mapsto 1/x$. Suppose that for each $1 \leq i \leq L - 1$, $D_i$ is a diagonal matrix of size $d_i \times d_i$ and $P_i$ is any $d_i \times d_i$ permutation matrix. Then the parameter map $\Psi_{\dd,\sigma}$ is invariant under the transformations
    \begin{align*}
        W_1 &\leftarrow P_1 D_1 W_1 \\
        W_2 &\leftarrow P_2 D_2 W_2 D_1 P_1^T \\
        &\ \vdots \\
        W_L &\leftarrow W_L D_{L-1} P_{L-1}^T.
    \end{align*}
    Consequently, the dimension of a generic preimage of $\Psi_{\dd,\sigma}$ is at least $\sum_{k=1}^{L-1} d_k$.
\end{lemma}
\begin{proof}
    See Appendix~\ref{apx:fiber-symmetries-general}
\end{proof}

Since the ambient space $\R(x_1,\dots,x_{d_0})$ is infinite dimensional, we fix the architecture $\dd$ and the corresponding degrees $n(\dd),m(\dd)$ (see Lemma~\ref{lemma:total-degree}) and define a \emph{combinatorial} parameter map that sends $\ww$ to the coefficients of the numerators $P_{i,\ww}$ and the denominator $Q_{\ww}$.

\subsection{The combinatorial parameter map}
\label{subsec:2-combinatorial-param-map}

We are interested in studying all possible tuples of rational functions that can be represented by the rational neural network $f_{\ww}$. Since all $f_{i,\ww}$'s share the same common denominator $Q_{\ww}$ (see Proposition \ref{theorem:neural-network-closed-form-general}), we take 
\[
    (S^{n(\dd)}(\R^{d_0}))^{d_L}\times S^{m(\dd)}(\R^{d_0})    
\]
as the ambient space. We identify the \emph{neuromanifold} $\mathcal{M}_{\dd,\sigma}$ with the image of the {\em combinatorial} parameter map
\begin{equation}
\label{eq:parameter-map-combinatorial-general}
    \Psi_{\dd,\sigma}:\R^{N}\to (S^{n(\dd)}(\R^{d_0}))^{d_L}\times S^{m(\dd)}(\R^{d_0}),\hspace{0.2cm} \ww\mapsto (P_{1,\ww},\dots,P_{d_L,\ww,},Q_{\ww}),
\end{equation}
where $N=\sum_{i=0}^{L-1}d_i\cdot d_{i+1}$ is the total number of parameters in $\ww$. This map is well-defined: if the polynomials $P_{1,\ww},\dots,P_{d_L,\ww}$ and $Q_{\ww}$ have a common factor, we do not cancel it out. We conjecture that such cancellations happen on a measure-zero subset of $\R^N$.

To use the tools of algebraic geometry, we introduce the main algebraic object of study.
\begin{definition}
    The \textit{neurovariety} $\mathcal{V}_{\dd,\sigma}$ is the Zariski closure of the neuromanifold $\mathcal{M}_{\dd,\sigma}$ in the space $(S^{n(\dd)}(\R^{d_0}))^{d_L}\times S^{m(\dd)}(\R^{d_0})$.
\end{definition} 
Since the map $\Psi_{\dd,\sigma}$ is polynomial, then the neuromanifold $\mathcal{M}_{\dd,\sigma}$ is a \emph{semialgebraic} set by the Tarski–Seidenberg theorem. The neurovariety $\mathcal{V}_{\dd,\sigma}$ being the Zariski closure of $\mathcal{M}_{\dd,\sigma}$ is an irreducible \emph{algebraic variety} \cite{KILEEL19}. In other words, $\mathcal{M}_{\dd,\sigma}$ can expressed as a finite union of subsets defined by polynomial equalities and inequalities, whereas $\mathcal{V}_{\dd,\sigma}$ is defined only by polynomials.

One of the central questions in algebraic machine learning \cite{KILEEL19,KUBJAS24} is whether the neurovariety associated with a given architecture fills the entire ambient space.
\begin{definition}
  The neurovariety $\mathcal{V}_{\dd,\sigma}$ is called \emph{filling} if it is equal to the ambient space, i.e., \[
      \mathcal{V}_{\dd,\sigma}
      \;=\;
      \left(S^{n(\dd)}(\mathbb{R}^{d_0})\right)^{d_L} \times S^{m(\dd)}(\mathbb{R}^{d_0}).
  \]
\end{definition} 

\begin{remark}
    We say that the architecture $\dd$ is \emph{filling} when its corresponding neurovariety is filling.
\end{remark}

In general, the neuromanifold may be much smaller than the neurovariety. If the neuromanifold is not equal to the ambient space, but its neurovariety fills the space, then the neuromanifold is called \textit{thick}. If the neuromanifold itself fills the ambient space, it is called \emph{filling}.

In contrast to polynomial neural networks \cite{KILEEL19,KUBJAS24}, changing the hidden dimensions $d_i$ in RationalNets alters the degree of the numerator and denominator of the output $f_{i,\ww}$. Therefore, we focus on the expressivity of a \emph{fixed} neural-network architecture $\dd=(d_0,d_1,\dots,d_L)$ via its combinatorial map.

A necessary condition for the neurovariety to be filling is that the network has enough parameters to cover the entire ambient space. In general, the number of parameters $N = \sum_{i=0}^{L-1} d_i d_{i+1}$ is much smaller than the dimension of the ambient space,
\begin{equation}
\label{eq:dim-ambient-space-formula}
    \dim\left(\left(S^{n(\dd)}(\R^{d_0})\right)^{d_L} \times S^{m(\dd)}(\R^{d_0})\right)
    = d_L \binom{d_0 + n(\dd)-1}{n(\dd)} + \binom{d_0 + m(\dd)-1}{m(\dd)},
\end{equation}
for a general architecture $\dd$. 

To determine which neurovarieties are filling, we must first identify the architectures for which the number of parameters exceeds the dimension of the ambient space. 

\begin{problem}
    Determine all architectures where the number of parameters exceeds the ambient dimension. Classify all filling architectures.
\end{problem}

Throughout this paper, we restrict our attention to two families of RationalNets
\begin{enumerate}
    \item \emph{Shallow networks} with architecture $\dd=(n,m,k)$ in Section \ref{section:shallow-networks} and Section \ref{sec:ideals}, and
    \item \emph{Deep binary networks} with architectures $\dd=(2,2,\dots,2,k)$ in Section \ref{section:binaryNN}. 
\end{enumerate}

For each family, we study the neurovariety $\mathcal{V}_{\dd,\sigma}$. Specifically, for both shallow networks and deep binary neural networks, we address the following questions
\begin{enumerate}
    \item \textbf{The membership problem.} 
        Characterize the rational functions that can be represented by a fixed architecture of the network $(\dd,\sigma)$ {and how to reconstruct the parameters}.
    \item \textbf{Filling architectures characterization.} 
        Identify all architectures $\dd$ for which the neurovariety $\mathcal{V}_{\dd,\sigma}$ is filling. Furthermore, show that there are no filling neuromanifolds $\mathcal{M}_{\dd,\sigma}$.
    \item \textbf{Model description.}  
         Provide algebraic description $\mathcal{V}_{\dd,\sigma}$ for several architectures.
\end{enumerate}

\section{Shallow neural networks}\label{section:shallow-networks}

The first family of neural networks we consider is the family of \emph{shallow neural networks}. 

\begin{definition}
    A \textit{shallow neural network} is a neural network with one hidden layer. We write the dimension vector of the  architecture as $\dd=(n,m,k)$.
\end{definition}

In Section \ref{sec:3-closed-form-expression}, we provide a closed-form expression of shalow networks and their connection with tensor decomposition.
In Sections \ref{sec:3-reconstructing-ww} and \ref{sec:3-algebraic-tensor-decomp} we provide two different methods for determining whether a function belongs to the neuromanifold corresponding to a shallow neural network.

\subsection{Closed Form Expression}
\label{sec:3-closed-form-expression}

For the remainder of this section, we denote the entries of $W_1$ and $W_2$ by $a_{ij}$ and $b_{ij}$, respectively. 
Let $\xx\in\R^n$, and define $\ell_i$ to be the linear form given by the $i$th coordinate of $W_1\xx$, i.e.,  {$\ell_i = (W_1\xx)_i= \sum_{j=1}^n a_{ij} x_j$}. For each $j=1,\dots,m$, we define
\[
    \hat{\ell}_{j,m} = \ell_1\dots \ell_{j-1}\ell_{j+1}\dots \ell_m.
\]
In other words, $\hat{\ell}_{j,m}$ is the product of all $\ell_i$ except for $\ell_j$. 

According to Theorem~\ref{theorem:neural-network-closed-form-general}, the $i$th entry of the output of $f_{\ww}(\xx)$ is 
\begin{equation}
\label{eq:nmk-general-output}
    f_{i,\ww}(\xx) = \frac{P_{i,\ww}(\xx)}{Q_{\ww}(\xx)}=\frac{b_{i1} \hat{\ell}_{1,m}+\dots+b_{im} \hat{\ell}_{m,m}}{\ell_1\ell_2\dots \ell_m}.
\end{equation}

Therefore, a point $(P_1,\dots,P_k,Q)$ belongs to the neuromanfold $\mathcal{M}_{\dd,\sigma}$ if and only if there exist parameters $\ww\in\R^N$ such that $P_i$ and $Q$ admit the decomposition above. 

This characterization can also be expressed in the language of tensors (see \cite{KOLDA2009} for more details on tensors).
Let $e_i$ be the $i$-th standard basis vector in $\mathbb{R}^m$, and let $\Sym(e_1 \otimes e_2 \otimes \dots \otimes e_m)$ be the symmetric tensor obtained by symmetrizing the rank-one tensor $e_1 \otimes e_2 \otimes \dots \otimes e_m$. In other words, 
\[
\Sym(e_1 \otimes e_2 \otimes \dots \otimes e_m) =\frac1{m!} \sum_{\pi}e_{\pi(1)} \otimes e_{\pi(2)} \otimes\cdots\otimes e_{\pi(m)},
\]
where $\pi$ ranges  over all permutations on $\{1,\ldots, m\}$.
We also define 
\[
    \hat{e}_{j,n} = e_1 \otimes \dots \otimes e_{j-1} \otimes e_{j+1} \otimes \dots \otimes e_n    
\]
to be the tensor formed by omitting the $j$-th factor in the outer product. 

If $T$ is a symmetric tensor of order $m$, then the contraction along all its modes with a vector $\xx\in\R^m$ is denoted by $T \circ \xx^{\otimes m}$. More precisely,
\[
T \circ \xx^{\otimes m} = \sum_{i_1,\ldots, i_m}T_{i_1i_2\ldots i_m}x_{i_1}x_{i_2}\cdots x_{i_m}.
\]
Observe that $T \circ \xx^{\otimes m}$ is a homogeneous polynomial of degree $m$ in the variables $x_1,\ldots x_n$.

\begin{remark}
    Throughout this paper, we often do not differentiate between a homogeneous polynomial and its associated symmetric tensor. 
\end{remark}

\begin{lemma}
\label{lemma:closed-tensor-form-shallow-network}
Let $\dd = (n, m, k)$ and $\ww = (W_1, W_2)$. Then the numerators $P_{i,\ww}$ and the denominator $Q_{\ww}$ are given by
\begin{equation}
\label{eq:shallow-network-equations}
    \begin{aligned}
        P_{i,\ww}(\xx) &= \left(\sum_{j=1}^m b_{ij} \Sym(\hat{e}_{j,m})\right) \circ (W_1 \xx)^{\otimes(m-1)}, \\
        Q_{\ww}(\xx) &= \Sym(e_1 \otimes \dots \otimes e_m) \circ (W_1 \xx)^{\otimes m}.
    \end{aligned}
\end{equation}
\end{lemma}
\begin{proof}
    See Appendix~\ref{apx:closed-tensor-form-shallow-network}
\end{proof}

\begin{example}
\label{example:shallow-(3,3,1)}
Consider the network $f_{\ww}$ with architecture $\dd = (3,3,1)$, given by the composition
\[
    \mathbb{R}^3 \xrightarrow{W_1} \mathbb{R}^3 \xrightarrow{\sigma} \mathbb{R}^3 \xrightarrow{W_2} \mathbb{R}.
\]
The output of the network is 
\[
    f_{\ww}(\xx) = W_2 \sigma(W_1 \xx) = 
    \begin{bmatrix}
        b_1 & b_2 & b_3
    \end{bmatrix}
    \begin{bmatrix}
        1/\ell_1 \\
        1/\ell_2 \\
        1/\ell_3
    \end{bmatrix}
    = \frac{b_1 \ell_2 \ell_3 + b_2 \ell_1 \ell_3 + b_3 \ell_1 \ell_2}{\ell_1 \ell_2 \ell_3} 
    = \frac{P_{1,\ww}(\xx)}{Q_{\ww}(\xx)}.
\]
Here the numerator is
\[
    P_{1,\ww}(\xx) = \left( b_1 \Sym(e_2 \otimes e_3) + b_2 \Sym(e_1 \otimes e_3) + b_3 \Sym(e_1 \otimes e_2) \right) \circ (W_1 \xx)^{\otimes 2},
\]
and the denominator is $Q_{\ww}(\xx) = \Sym(e_1 \otimes e_2 \otimes e_3) \circ (W_1 \xx)^{\otimes 3}.$
\end{example}

\begin{remark}
    If $d=(n,n,k)$, then $W_1\in\mathbb R^{n\times n}$ is generically invertible. Hence, understanding the neuromanifold $\mathcal{M}_{\dd,\sigma}$ is equivalent to studying the orbit of the tuple of symmetric tensors 
    \[
        (P_1,\dots,P_k,Q)=\left(\sum_{j=1}^m b_{1j} \Sym(\hat{e}_{j,m}),\dots,\sum_{j=1}^m b_{nj} \Sym(\hat{e}_{j,m}),\Sym(e_1 \otimes \dots \otimes e_m)\right)
    \]
    under the $GL_n(\R)$ action $(A\cdot T)(\xx) := T\circ (A\xx)^{\otimes m}$ for $A\in GL_n(\mathbb R)$ and $T\in S^m(\mathbb R^n)$.
\end{remark}

\subsection{Algorithm for recovering the parameters \texorpdfstring{$\ww$}{ww}}
\label{sec:3-reconstructing-ww}
We now present a method for reconstructing the parameters $\ww$ from the output function $f_{\ww}$. We first recover the matrix $W_1$, and then the matrix $W_2$.

 Consider 
 the neuromanifold $\mathcal{M}_{\dd,\sigma}(\C)$,  which is the image of the map $\Psi_{\dd, \sigma}$, where the parameters $\ww = (W_1,W_2)$ are allowed to be complex. 
 
 Assume we are given a tuple in the ambient space $(P_1,\dots,P_k,Q)\in \left(S^{m-1}(\C^n)\right)^k \times S^m(\C^n)$. According to Equation~\eqref{eq:nmk-general-output}, the tuple $(P_1,\dots,P_k,Q)$ belongs to the neuromanifold $\mathcal{M}_{\dd,\sigma}(\C)$ if and only if there exist parameters $\ww=(W_1,W_2)$ such that the following system of equations holds
\begin{equation}
\label{eq:nmk-definining-system}
    \begin{aligned}
        P_{i}(\xx) &= b_{i1} \hat{\ell}_{1,m}+\dots+b_{im} \hat{\ell}_{m,m}, \\
        Q(\xx) &= \ell_1\dots \ell_m,
    \end{aligned}
\end{equation}
where $\ell_j$ denotes the linear form given by $(W_1\xx)_j$ and $\hat{\ell}_{j,m}$ is the product of all $\ell_i$ except $\ell_j$.

We can recover the parameters $\ww$ from this system via the following algorithm.

\begin{algorithm}
\caption{Weight reconstruction for architecture $(n,m,k)$}
\label{algorithm:shallow-weights-reconstruction}
\begin{algorithmic}[1]
    \State \textbf{Factor test:} Check if the denominator $Q$ factors into a product of $m$ linear forms. 
        \Statex \hspace{\algorithmicindent} \textit{If not, then output $(P_1,\dots,P_k,Q)\not\in\mathcal{M}_{\dd,\sigma}(\C)$ and halt.}
    \State \textbf{Recover $W_1$:} Find linear forms $\ell_1,\dots,\ell_m$ such that $Q(\xx)=\ell_1\ell_2\dots \ell_m$.
    \State \textbf{Recover $W_2$:} With $W_1$ known, solve the linear system \eqref{eq:nmk-definining-system} for the coefficients $b_{ij}$. 
        \Statex \hspace{\algorithmicindent} \textit{If the system is inconsistent, output $(P_1,\dots,P_k,Q)\not\in\mathcal{M}_{\dd,\sigma}(\C)$.}
\end{algorithmic}
\end{algorithm}
\noindent
Let us now expand on each individual step of Algorithm~\ref{algorithm:shallow-weights-reconstruction}, outlining the methods used and how each step can be implemented in practice.\\

\noindent
\textbf{Factor test:} The first step of Algorithm \ref{algorithm:shallow-weights-reconstruction} checks whether $Q$ can be factorized into the product of $m$ linear forms. This will allow us to reconstruct the rows of $W_1$. This step is equivalent to checking whether $Q$ is in the Chow variety (see Section~\ref{sec:3-chow-variety}), which is defined by the vanishing of the so called Brill's equations~\cite{GUAN18}. 

The construction of Brill's equations is given in Chapter $4$ in~\cite{GELFAND94} where they are defined as the coefficients of a certain trilinear form $B(x,y,z)$.

\begin{theorem}[\cite{GELFAND94}]
    If $f\in S^{m}(\C^n)$, then $f$ splits into the product of linear factors if and only if $B(x,y,z)=0$ for all $x,y,z$, where the coefficients  of  $B(x, y, z)$ are computed from the coefficients of $f$.
\end{theorem}

While Brill's equations work well for determining if a form splits into the product of linear factors, their computation quickly becomes infeasible for large values of $m$ and $n$ (see Example~\ref{ex:331-Chow-variety}). For the architectures $\dd=(n,2,k)$, we present in Section~\ref{sec:4-nmk-architecture} an efficient test to determine whether $(P_1,\dots,P_k,Q)$ belongs to the neuromanifold $\mathcal{M}_{\dd,\sigma}(\C)$.\\

\noindent
\textbf{Recovering $W_1$:} To reconstruct $W_1$, we first need to determine the linear forms $\ell_1,\ldots, \ell_m$. There are two main methods to reconstruct them: deterministic and probabilistic. For the probabilistic approach see~\cite{KOIRAN18}, where it is shown that in the case of the $\dd=(n,n,1)$ architecture, one can efficiently recover the linear forms into which $Q(\xx)$ splits. For the deterministic approach, according to Proposition $2.11$ in~\cite{GELFAND94}, the matrix $W_1$ can be recovered by solving the following system of equations.

\begin{proposition}[\cite{GELFAND94}]
\label{lemma:lin-form-divides-hom-pol}
    Let $Q\in S^{d}(\C^n)$. Then, the nonzero linear form $l$ in $(\C^n)^*$ divides $Q$ if and only if 
    \begin{equation}
    \label{eq:Gelfand's-method-reconstructing-W1}
       \sum_{k=0}^d\frac{(-1)^k}{k!}\Delta_k(Q(x))\ell(x)^k\ell(y)^{d-k}=0, 
    \end{equation}
    where $\Delta_k(Q(x))=(\sum_{i=0}^ny_i\frac{\partial}{\partial x_i})^kQ(x)$.
\end{proposition}

Below, let us discuss two examples of reconstrucing $W_1$.

\begin{example}
    Let $Q(\xx)=x_1^3 - x_1x_2^2 - x_1x_2x_3 + x_2^2x_3 - x_1x_3^2 + x_2x_3^2$ and $\ell(\xx)=c_1x_1+c_2x_2+c_3x_3$. Setting up the system~\eqref{eq:Gelfand's-method-reconstructing-W1} above will lead us to $54$ polynomial equations. The Gröbner basis of the ideal $I$ generated by these equations is equal to
    \begin{enumerate}
        \begin{multicols}{2}
            \item $c_1c_3^3 - 2c_2c_3^3 + c_3^4$,
            \item $c_1^3 + c_2^3 - 3c_2c_3^2 + c_3^3$,
            \item $c_1^2c_2 - c_3^3$,
            \item $c_1c_2^2 + c_2^3 - c_1c_3^2 - c_3^3$,
            \item $c_1^2c_3 - c_3^3$,
            \item $c_1c_2c_3 - \frac{1}{2}c_1c_3^2 - \frac{1}{2}c_3^3$,
            \item $c_2^2c_3 - c_2c_3^2$.
        \end{multicols}
    \end{enumerate}

    and the primary decomposition of $I$ is equal to 
    \[
        I=(c_3, c_1 + c_2)\cap (c_2 - c_3, c_1 - c_3)\cap (c_2, c_1 + c_3) \cap J
    \]
    where $J=(c_3^3, c_2c_3^2, c_2^2c_3, 2c_1c_2c_3 - c_1c_3^2, c_1^2c_3, c_2^3, c_1c_2^2 - c_1c_3^2, c_1^2c_2, c_1^3)$ is primary and $\sqrt{J}=(c_1, c_2, c_3)$. The first three primary ideals give us exactly that
    \[
        f(x) = (x_1+x_2+x_3)(x_1-x_2)(x_1-x_3),\text{ so }W_1=\begin{bmatrix}
        1 & 1 & 1\\
        1 & -1 & 0\\
        1 & 0 & -1
    \end{bmatrix}.
    \]
\end{example}

However, Galois theory tells us that if we get polynomials of degree strictly higher than four, then the solution might not be as nice as in the previous example. If we take $Q$ to be a homogeneous polynomial of degree five over two variables, then there are no closed-form equations that describe the linear forms $\ell_i$ in the decomposition of $Q$. The next example illustrates that we can hope to reconstruct the forms only numerically at best.

\begin{example}
\label{example:251}
    Let $\dd=(2,5,1),$ and $Q(\xx)=x_1^5 - x_1x_2^4 + x_2^5$. Note that $Q(\xx)$ is totally decomposable since it is a binary form. However, to find the five linear forms, we need to solve six equations in~(\ref{eq:Gelfand's-method-reconstructing-W1}). Using the Gröbner basis, we will obtain the single equation
    \[
        c_1^5 - c_1c_2^4 - c_2^5 = 0.
    \]
    To obtain the five desirable linear forms, we need to solve this equation. However, the equation above can be solved only using numerical methods as there is no closed-form solution according~\cite{LANG94}. Using numerical techniques, we can obtain the rows of $W_1$ up to a desired precision
    \[
        W_1^T\approx
        \begin{bmatrix}
            1 & 1 & 1 & 1 & 1 \\
            -0.8566 & -0.1500-0.8974 i & -0.1500+0.8974 i & 1.0783-0.4969 i & 1.0783+0.4969 i
        \end{bmatrix}^T.
    \]
\end{example}

\begin{remark}
    An approach that combines the first two steps of the algorithm by solving a univariate polynomial and a linear system instead is discussed in Section~\ref{sec:3-algebraic-tensor-decomp}.
\end{remark}

\noindent
\textbf{Recovering $W_2$:} 
Now  let us assume that we have reconstructed $W_1$. The next and final step of reconstructing $W_2$ can be done by solving the $k$ linear equations in $W_2$ from~\eqref{eq:nmk-definining-system}
\[
    P_i = b_{i1} \hat{\ell}_{1,m}+\dots+b_{im} \hat{\ell}_{m,m}.
\]
The polynomials $\{\hat{\ell}_{1,m},\dots,\hat{\ell}_{m,m}\}$ can be obtained from $W_1$ and span at most an $m$-dimensional linear subspace in $S^{m-1}(\C^n)$. Therefore, we can reconstruct all the rows of $W_2$ if and only if $P_i$ belongs to the linear span of $\{\hat{\ell}_{1,m},\dots,\hat{\ell}_{m,m}\}$.

\subsection{A general method for finding the matrix \texorpdfstring{$W_1$}{W1}}
\label{sec:3-algebraic-tensor-decomp}

In this subsection we describe an alternative 
algebraic method to find the matrix $W_1$  
for $(n,m,k)$ architectures.
The advantage of the method we propose here,
is that 
the only computational difficulty lies in finding the roots of a univariate polynomial of degree $m$. 
This can be done by radicals for all polynomials up to degree $4$.
For all cases not solvable by radicals one can find solutions with arbitrary precision using numerical techniques.  
After this nonlinear part, 
the remainder of the problem boils down to the solution of a linear system of equations. 

Given a polynomial $Q(\xx)$,
we want to decompose it into a product 
\[
    \prod_{i=1}^m(x_1+a_{i2}x_2+ \cdots +a_{in}x_n).
\]
Instead  of first checking if $Q(x)$ can be factored into linear forms, we directly try to find the factorization.
We assume that the coefficient of $x_1^m$ is nonzero and therefore it can be normalized to $1$. This  assumption is not restrictive 
except for polynomials where (up to permutation) no $m$th power of an $x_i$ appears. 
Under this assumption we can set  the coefficient of $x_1$ in all the linear forms in the product to $1$.
This solves the ambiguity of finding the linear forms up to a multiplicative constant and creates unique solutions. 

Comparing coefficients in the polynomial equality 
\[
x_1^m+\sum_{\substack{u \in \N^n:\deg(u)=m \\ u \neq (m,0,\dots,0)}}C_ux^u = \prod_{i=1}^m(x_1+a_{i2}x_2+ \cdots +a_{in}x_n)
\]
we obtain in particular
\begin{align*}
    a_{12}+a_{22}+ \cdots + a_{n2} &= C_{(m-1,1,0,\dots ,0)} \\
    a_{12}a_{22}+a_{12}a_{32} +\cdots + a_{(n-1)2}a_{n2} &= C_{(m-2,2,0,\dots ,0)} \\
    &\vdots \\
    \mu_k(a_{12},a_{22}, \dots, a_{n2}) &= C_{(m-k,k,0,\dots ,0)} \\
    &\vdots \\
    a_{12}a_{22} \cdots a_{n2} &= C_{(0,m,0,\dots ,0)}
\end{align*}
In general the coefficient of $C_{(m-k,k,0,\dots ,0)}$ for $k=1,\dots , m$ 
is the $k$-th symmetric polynomial in $a_{12},a_{22}, \dots, a_{n2}$, 
usually denoted by $\mu_k(a_{12},a_{22}, \dots, a_{n2})$. Then for the univariate polynomial
\[g(y):=y^m + \sum_{k=1}^{m-1}(-1)^kC_{(m-k,k,0,\dots ,0)}y^k\]
we obtain
\begin{align*}
    g(y)&=y^m + \sum_{k=1}^{m-1}(-1)^k\mu_k(a_{12},a_{22}, \dots, a_{n2})y^k \\
    &=\prod_{k=1}^m(y-a_{k2}),
\end{align*}
so we can recover all $a_{k2}$ by solving the equation $g(y)=0$. 

Knowing the coefficients $a_{12},\ldots, a_{n2}$ (up to permutation),
allows us to generically recover the remaining coefficients by solving linear systems of equations.
In particular,
in order to recover $a_{1l}, \dots, a_{nl}$
we look at the $m$ coefficients 
\[C_{(m-1,0,0,\dots,0,1,0,\dots, 0)},C_{(m-2,1,0,\dots,0,1,0,\dots, 0)},\dots, C_{(0,m-1,0,\dots,0,1,0,\dots, 0)},\]
where after the second position the index of $C$ has a $1$ in position $l$ and $0$ elsewhere. 
We then get
\begin{align*}
    a_{1l}+a_{2l}+ \cdots + a_{ml} &= C_{(m-1,0,0,\dots,0,1,0,\dots, 0)} \\
    a_{1l}\mu_1(\hat{a}_{12})+a_{2l}\mu_1(\hat{a}_{22})+ \cdots + a_{ml}\mu_1(\hat{a}_{m2}) &= C_{(m-2,1,0,\dots,0,1,0,\dots, 0)} \\
    & \vdots \\
    a_{1l}\mu_{n-1}(\hat{a}_{12})+a_{2l}\mu_{n-1}(\hat{a}_{22})+ \cdots + a_{n2}\mu_{n-1}(\hat{a}_{m2}) &= C_{(0,m-1,0,\dots,0,1,0,\dots, 0)},
\end{align*}
where we use the symbol $\hat{a}_{i2}$ to denote the vector $(a_{12},a_{22},\dots a_{m2})$ with entry $a_{i2}$ removed.
For a fixed $l$, this is a linear system in the $a_{il}$ that generically has one solution,
so we can recover the rest of the parameters $a_{ij}$ uniquely.

In the end, it is necessary to check if the product of the linear forms we obtained is indeed equal to the polynomial $Q(x)$.
If not, we know that the original polynomial was not decomposable to begin with.
If yes, then we know that the decomposition is real exactly when all the $a_{ij}$ are real. 
If not all of them are real, then no such decomposition exists.

\section{Algebraic geometry of shallow neural networks}
\label{sec:ideals}
In this section, we study shallow neural networks through the lens of algebraic geometry. 
The main goal is to describe the neurovariety, 
i.e., to provide a generating set of polynomials for the corresponding ideal. In Section \ref{sec:3-parameters-count},  we show that the architectures $\dd=(2,m,k)$ are the only filling shallow architectures.
In Section~\ref{sec:4-nmk-architecture}, we fully describe the neurovariety for architectures $\dd=(n,2,k)$ and provide a partial description for general architectures $\dd=(n,m,k)$. In Section \ref{sec:3-chow-variety}, we discuss the connection between shallow networks and Chow varieties.

\subsection{Filling shallow architectures}
\label{sec:3-parameters-count}

Let $M(n,m,k)$ be the dimension of the ambient space of the shallow network $\dd=(n,m,k)$. According to Equation \eqref{eq:dim-ambient-space-formula},
\begin{equation}
\label{eq:dim-amb-space-shallow-networks}
    M(n,m,k) = \binom{n+m-1}{m} + k\binom{n+m-2}{m-1}.
\end{equation}
On the other hand, the number of trainable parameters is $N(n,m,k)=mn+km$. To determine for which triples $(n,m,k)$ the corresponding neurovariety $\mathcal{V}_{\dd,\sigma}$ is filling, we must identify all shallow architectures for which the parameter count is at least the dimension of the ambient space. The following lemma shows that the only non-trivial architectures ($n>1$) that satisfy this parameter constraint are of the form $\dd=(2,m,k)$.

\begin{lemma}
\label{lemma:shallow-network-params-vs-amb-dim}
  Let $N(n,m,k)=mn+km$ and $M(n,m,k)=\binom{n+m-1}{m}+k\binom{n+m-2}{m-1}.$
  Then
  \[
     N(n,m,k)\;\ge\;M(n,m,k)
     \quad\Longleftrightarrow\quad
     n=1,2.
  \]
\end{lemma}
\begin{proof}
    See Appendix~\ref{apx:shallow-network-params-vs-amb-dim}
\end{proof}

\begin{remark}
    The architecture $\dd=(1,m,k)$ is trivial. The output of the network is equal to $P_{i,\ww}(x)=(\sum_{j=1}^mb_{ij}\hat{a}_{j1})x^{m-1}$ and $Q_{\ww}(\xx)=(a_{11}\dots a_{m1})x^m$. Thus, the corresponding neuromanifold $\mathcal{M}_{\dd,\sigma}$ fills the entire ambient space as the ambient space is spanned by the monomials $x^{m-1}$ and $x^m$. So, the coefficients in front of them can independently take any real value.
\end{remark}

Next, we show that when $n=2$, and its neurovarity $\mathcal{V}_{\dd,\sigma}(\C)$ is filling, but the neuromanifold $\mathcal{M}_{\dd,\sigma}(\C)$ is not. We 
 will use the following Lemma.

\begin{lemma}
\label{lemma:lin-independence-of-g_hat_i}
    The forms $\{\hat{\ell}_1,\hat{\ell}_2,\dots,\hat{\ell}_m\}$ are linearly independent in $S^{m-1}(\C^n)$ if and only if all rows of $W_1$ are pairwise linearly independent,
\end{lemma}
\begin{proof}
    See Appendix~\ref{apx:lin-independence-of-g_hat_i}
\end{proof}

\begin{proposition}
\label{prop:neuromanifold-(2,m,k)}
    If $\dd=(2,m,k)$ with $m\geq2$ and $k\geq1$, then the neurovariety $\mathcal{V_{\dd,\sigma}}(\C)$ is filling, but the neuromanifold $\mathcal{M}_{\dd,\sigma}(\C)$ is not.
\end{proposition}
\begin{proof}
    See Appendix~\ref{apx:neuromanifold-(2,m,k)}
\end{proof}

\subsection{Architectures \texorpdfstring{$\dd=(n,m,k)$}{d=(n,m,k)}}
\label{sec:4-nmk-architecture}

We have already discussed architectures of the form $\dd = (n,1,m)$  in Section~\ref{subsec:2-shape-of-rational-neural-networks}.
The first interesting case is when the middle layer has dimension $2$. 
Let $\dd = (n,2,m)$ with $n\geq 2$.
Then the neural network
has the form 
\begin{equation*}
    \Bigg(\frac{\sum_{1\leq i \leq n}C_{k,\e_i}x_i}{\sum_{1\leq i \leq j \leq n}C_{\e_i+\e_j}x_ix_j} \Bigg)_{k=1, \dots , m}.
\end{equation*}
In this case we can fully characterize the ideal of the neurovariety. 
\begin{theorem}
\label{theorem:n2m-architecture-ideal}
    Let $\dd = (n,2,m)$ with $n\geq 2$. Then the ideal of the neural network is generated by all $3\times 3$ minors of the matrix
\[\mathbf{M} = 
        \bordermatrix{~ & 1,\mathbf{0} & \dots & m,\mathbf{0}& \e_1 & \e_2 &\dots & \e_{n}\cr
              \e_1 & C_{1\e_1} & \dots &C_{m,\e_1} & 2C_{2\e_1} & C_{\e_1+\e_2} & \dots &C_{{\e_1 +\e_{n}}}\cr
              \e_2 & C_{1,\e_2} & \dots &C_{m,\e_2} & C_{\e_1+\e_2} & 2C_{2\e_2} & \dots &C_{\e_2 +\e_{n}} \cr
              \vdots & \vdots & \ddots & \vdots &\vdots & \vdots & \ddots & \vdots\cr
              \e_{n} & C_{1,\e_{n}} & \dots &C_{m,\e_n}& C_{{\e_1 +\e_{n}}} & C_{{\e_2 +\e_{n}}} & \dots &2C_{2\e_{n}} \cr
              }.
    \]
    The corresponding index of  the coefficient $C$
    is the sum of the row and column indexes,
    with the exception that if the indices agree then the $C$ coefficient is multiplied by a factor of $2$.

    The dimension of the ideal is $2(n+m)-1$.
\end{theorem}
\begin{proof}
    See Appendix~\ref{apx:n2m-architecture-ideal}
\end{proof}

\begin{example}
    Consider $n=3$, $m=1$. 
    Then the neural network has the form
    \begin{equation*}
        f(\xx) = \frac{C_{\e_1}x_1+C_{\e_2}x_2+C_{\e_3}x_3}{C_{2\e_1}x_1^2+C_{2\e_2}x_2^2+C_{2\e_3}x_3^2+C_{\e_1+\e_j2}x_1x_2+C_{\e_1+\e_3}x_1x_3+C_{\e_2+\e_3}x_2x_3},
    \end{equation*}
    and the corresponding matrix is
    \[\mathbf{M} = 
        \bordermatrix{~ & \mathbf{0} & \e_1 & \e_2 & \e_{3}\cr
              \e_1 & C_{\e_1} & 2C_{2\e_1} & C_{\e_1+\e_2} & C_{{\e_1 +\e_{3}}}\cr
              \e_2 & C_{\e_2} & C_{\e_1+\e_2} & 2C_{2\e_2} & C_{\e_2 +\e_{3}} \cr
              \e_{3} & C_{\e_{3}} & C_{{\e_1 +\e_{3}}} & C_{{\e_2 +\e_{3}}} & 2C_{2\e_{3}} \cr
              }.
    \]
    Theorem~\ref{theorem:n2m-architecture-ideal} shows that $\mathbf{M}$ drops rank if and only if the function $f(\xx)$ arises from the neural network (i.e., the numerator is the product of two linear forms and the denominator is a linear combination of these forms).
\end{example}

The parametrization for $\dd=(d_0,3,1)$ is as follows:
\begin{align*}
    C_{3\e_i} &\mapsto a_{1i}a_{2i}a_{3i} \\
    C_{2\e_i + \e_j} &\mapsto a_{1i}a_{2i}a_{3j} +a_{1i}a_{2j}a_{3j}+a_{1i}a_{2i}a_{3i}\\
    C_{\e_i + \e_j +\e_k} &\mapsto \sum_{\pi \in S_3(i,j,k)}a_{1\pi(1)}a_{2\pi(2)}a_{3\pi(3)}.
\end{align*}

The matrix $\mathbf{M}$ created similarly to the one above
with rows indexed by pairs $(\e_i,\e_j)$ for $1\leq i\leq j \leq n$
and columns indexed by $\e_i$ for $0\leq i \leq n$
has rank $3$.
Indeed, it can be shown that $\phi(\mathbf{M})$ can be written as a product of two matrices
of sizes $6 \times 3$ and $3 \times 4$.
We generalize this for all architectures in the following result:

\begin{proposition}
\label{prop:minorideals}
    Let $\dd = (d_0,d_1,d_2)$ be an architecture with $d_1 \geq 2$.
    We define a matrix $\mathbf{M}$, with 
    \begin{itemize}
        \item rows indexed by $(d_1-1)$-tuples $(\e_{i_1},\dots, \e_{i_{d_1-1}})$ of unit vectors of length $d_0$,
        where $1 \leq i_1 \leq \dots \leq i_{d_1-1} \leq d_0$;
        \item columns indexed either by $(k,\mathbf{0})$ for $k = 1, \dots , d_2$ corresponding to numerators of entries in the parametrization, or vectors $\e_j$ as above.
    \end{itemize}
    The entry of the matrix in position indexed by 
    $(\e_{i_1},\dots, \e_{i_{d_1-1}})$ and $(k,\mathbf{0})$ is 
    \begin{equation*}
        \frac{d_1!}{|S_{d_1-1}(\{i_1,\dots , i_{d_1-1}\})|!} C_{k,\e_{i_1}+\cdots +\e_{i_{d_1-1}}}.
    \end{equation*}
    and the entry of in position indexed by $(\e_{i_1},\dots, \e_{i_{d_1-1}})$ and $\e_j$ is 
    \begin{equation*}
        \frac{d_1!}{|S_{d_1}(\{i_1,\dots , i_{d_1-1}, j\})|!} C_{\e_{i_1}+\cdots +\e_{i_{d_1-1}}+\e_j},
    \end{equation*}
    where $S$ is the set of all permutations of a multiset.
    The matrix $\mathbf{M}$ has rank $d_1$.
\end{proposition}
\begin{proof}
    See Appendix~\ref{apx:minorideals}
\end{proof}

The equations obtained in the above proposition are not enough to fully describe the model when $d_1 \geq 3$ even for small architectures.
In fact, even adding information about the denominators still doesn't give the full picture,
as the following example shows. 

\begin{example}
    Consider the architecture given by $\dd = (3,3,1)$.
    Computing the whole ideal by means of elimination implemented in \texttt{Macaulay2} does not finish after several hours of running,
    so we try to find equations differently. 
    Computing the Jacobian coming from the parametrization, 
    we find that it has (maximal) rank $10$,
    so the ideal we are looking for has dimension $10$.
    
    Proposition \ref{prop:minorideals} implies that all $4$-minors of the matrix
    \begin{equation*}
        \mathbf{M} = 
        \bordermatrix{~ & \mathbf{0} & \e_1 & \e_2 & \e_{3}
              \cr
              (\e_1,\e_1) & 2C_{2\e_1} & 6C_{3\e_1} & 2C_{2\e_1+\e_2} & 2C_{{2\e_1 +\e_{3}}}
              \cr
              (\e_1,\e_2) & C_{\e_1+\e_2} & 2C_{2\e_1+\e_2} & 2C_{\e_1+2\e_2} & C_{{\e_1 +\e_2+\e_{3}}}
              \cr
              (\e_1,\e_3) & C_{\e_1+3\e_3} & 2C_{2\e_1+\e_3} & C_{\e_1+\e_2+\e_3} & 2C_{{\e_1 +2\e_{3}}}
              \cr
              (\e_2,\e_2) & 2C_{2\e_2} & 2C_{\e_1+2\e_2} & 6C_{3\e_2} & 2C_{{2\e_2+\e_{3}}}
              \cr
              (\e_2,\e_3) & C_{\e_2+\e_3} & C_{\e_1+\e_2+\e_3} & 2C_{2\e_2+\e_3} & 2C_{{\e_2+2\e_{3}}}
              \cr
              (\e_3,\e_3) & 2C_{2\e_3} & 2C_{\e_1+2\e_3} & 2C_{\e_2+2\e_3} & 6C_{{3\e_{3}}}
               \cr
              }
    \end{equation*}
    vanish.
    However, the ideal generated by those minors has dimension $13$, 
    so the minors do not give the full picture here. 

    We now add more equations by taking a closer look at the denominator.
    Taking the coefficients of the denominator only and using elimination
    gives rise to a set of $35$ equations.
    These are equivalent to Brill's equations.
    The ideal consisting of the minors of $\mathbf{M}$ and the generators of the ideal of the denominator has dimension $12$
    which means that there are still generators that are unaccounted for. 

    We finally resort to using the \texttt{MultigradedImplicitization} package for \texttt{Macaulay2},
    which finds that the ideal of polynomials of degree up to $4$ has the correct dimension $10$.
    This ideal is minimally generated by $185$ homogeneous polynomials of degree $4$.
\end{example}

\subsection{Shallow Neural Networks and Chow varieties}
\label{sec:3-chow-variety}

In general, any shallow rational neural network can be interpreted through the lens of Chow varieties, as we discuss in the next subsection.
Let $\zz=(z_1,\dots,z_k)$ and $\ell_j=(W_1\xx)_j$. Define the polynomial
\begin{equation}
\label{eq:H-polynomial}
    H(\xx,\zz)=\prod_{j=1}^m(b_{1j}z_1+\dots+b_{kj}z_k+\ell_j).
\end{equation}
Observe that $H$ is a homogeneous polynomial of degree $m$ over the $n+k$ variables $(\xx,\zz)$. If we expand $H$ and treat it as a polynomial in the $z$-variables, with coefficients that are polynomials in the $x$-variables, then 
\begin{align}\label{eq:H-coefficients}
    H(\xx,\zz) = Q_{\ww}(\xx) + \sum_{i=1}^{m}z_iP_{i,\ww}(\xx)+\sum_{i=1}^m\sum_{j=1}^mz_iz_jH_{ij}(\zz,\xx),
\end{align}
where each $H_{ij}$ is a polynomial of degree $m-2$. Here the constant term $Q_{\ww}$ is the denominator of the network output, and the coefficients of $z_i$ are precisely the numerators $P_{i,\ww}$. 

This allows us to establish the connection between shallow networks and the Chow variety: if a point $(P_1,\dots,P_k,Q)$ belongs to the neuromanifold $\mathcal{M}_{\dd,\sigma}$, then $H(\xx,\zz)$ splits into the product of linear forms according to \eqref{eq:H-polynomial}. 

This discussion implies the following.

\begin{lemma}
\label{lemma:chow-variety-shallow-networks}
    Let $\dd=(n,m,k)$. Then a point $(P_1,\dots,P_k,Q)$ belongs to the neuromanifold $\mathcal{M}_{\dd,\sigma}$ if and only if there exist polynomials $H_{ij}(\xx, \zz)$ so that the polynomial $H(\xx,\zz)$ defined in~\eqref{eq:H-coefficients} factors into the product of linear forms.
\end{lemma}

This observation shows that shallow neural networks are connected to Chow varieties which describe precisely the polynomials that factor into linear forms~\cite{GELFAND94}.

Brill, Gordon, and others obtained set-theoretic equations for the Chow variety, known as Brill's equations, which are computed in \cite{GUAN18}.
Although the number and degree of Brill's equations grows rapidly with $d$ and $n$, they remain useful for describing the neurovariety we study below.

However, Brill's equations involve all coefficients of $H$, including those of higher-order terms in the $z$-variables. To obtain equations involving only the coefficients of $P_i$ and $Q$, we must eliminate the ``non-essential'' variables. This corresponds to the fact that our variety is a projection of the Chow variety onto the coefficients of $Q(\xx)$ and  $P(\xx)$. Due to the computational complexity of Gröbner bases, this approach becomes infeasible for architectures beyond small or trivial cases. 

\begin{example}
\label{ex:331-Chow-variety}
    Following Example~\ref{example:shallow-(3,3,1)}, the polynomial $H$ corresponding to the architecture $\dd=(3,3,1)$ has the following form
    \[
        H(\xx,\zz)=Q(\xx) + P(\xx)z+(u_1x_1+u_2x_2+u_3x_3)z^2 +u_4z^3.
    \]
    For $H(\xx,\zz)$ to split into the product of three linear forms, we need to check that all $875$ Brill's equations vanish~\cite{BRIAND10}. These equations involve a total of $20$ variables among which we need to eliminate $4$ variables to obtain equations purely in terms of the 
 coefficients of $P$ and $Q$.
\end{example}

\section{Binary neural networks}
\label{section:binaryNN}

A \emph{binary deep neural network} is a neural network whose architecture is of the form $\dd=(2,\dots,2,d_L)$. In Section~\ref{sec:5-closed-form}, we provide a closed-form expression for the output of binary rational neural networks. In Section~\ref{sec:5-filling-archtiectures}, we classify all binary rational neural network architectures that are filling. 

\subsection{Closed Form Expression}
\label{sec:5-closed-form}

Let $\ww=(W_1,\dots,W_L)$ be the parameters of the network, and let $P_{12}$ be the $2\times 2$ permutation matrix corresponding to the transposition $(1,2)$. When the architecture of the neural network has the form $\dd=(2,\dots,2,d_L)$, the intermediate polynomials $p^{(i)}$ and $q^{(i)}$ in the recursive factorization from~Theorem~\ref{theorem:neural-network-closed-form-general} are linear and quadratic forms respectively.

For $i=2,\dots,L$, define the quadratic forms
\begin{equation}
\label{eq:def-quadratic-forms-binary-networks}
    q_{i+1}(\xx) \coloneqq \frac{1}{2}\xx^T(W_iP_{12}W_{i-1}\dots P_{12}W_1)^TP_{12}(W_iP_{12}W_{k-1}\dots P_{12}W_1)\xx,
\end{equation}
and set $q_0(\xx)=q_1(\xx)=1$. Define the vector of linear forms
\begin{equation}
\label{eq:vector-of-linear-forms-binary-networks}
    p_i(\xx)\coloneqq\left(W_iP_{12}W_{i-1}\dots P_{12}W_1\right)\xx    
\end{equation}
and denote its $j$th entry by $p_{j,i}$. Then the output of the neural network $f_{\ww}$ can be rewritten in closed form in terms of $p_L$ and $q_i$'s as described below.

\begin{proposition}
\label{prop:closed-form-binary-deep-network}
    Let $f_{\ww}:\R^{2}\to\R^{d_L}$ be the neural network with architecture $\dd=(2,2,\dots,2,d_L)$. Then $f_{\ww}(x)$ is a $d_L$-tuple of rational functions
    \[
        f_{\ww}(\xx)=\left(\frac{P_{1,\ww}(\xx)}{Q_{\ww}(\xx)},\dots,\frac{P_{d_L,\ww}(\xx)}{Q_{\ww}(\xx)}\right)^\top,
    \]
    where $P_{i,\ww}$ and $Q_{\ww}$ are homogeneous polynomials given by
    \begin{equation}
    \label{eq:RNN-coordinates-formulas-binary}
    \begin{aligned}
        P_{i,\ww}(\xx) &= p_{i,L}(\xx)q_{L-1}(\xx)q_{L-3}(\xx)\dots q_{\delta(L+1)}(\xx)\\
        Q_{\ww}(\xx) &= q_{L}(\xx)q_{L-2}(\xx)\dots q_{\delta(L)}(\xx)
    \end{aligned}
    \end{equation}
    with $\delta(L)=0$ if $L$ is even and $\delta(L)=1$ if $L$ is odd.
\end{proposition}
\begin{proof}
    See Appendix~\ref{apx:closed-form-binary-deep-network}
\end{proof}

\begin{example}
    Consider the network $f_{\ww}(\xx)$ with architecture $\dd=(2,2,2,1)$. The output of $f_{\ww}(\xx)$ is equal to
    \[
        f_{\ww}(\xx)=W_3\sigma_2 W_2 \sigma_1 W_1\xx = \frac{\ell_1\ell_2\left((c_1b_{22}+c_2b_{12})\ell_1+(c_1b_{21}+c_2b_{11})\ell_2\right)}{b_{12}b_{22}\ell_1^2+(b_{11}b_{22}+b_{12}b_{21})\ell_1\ell_2+b_{11}b_{21}\ell_2^2}=\frac{p_{1,3}(\xx)q_1(\xx)}{q_{2}(\xx)},
    \]
    where $q_1(\xx)=\ell_1\ell_2=\xx^TW_1P_{12}W_1\xx,$
    \[
        q_2(\xx) = b_{12}b_{22}\ell_1^2+(b_{11}b_{22}+b_{12}b_{21})\ell_1\ell_2+b_{11}b_{21}\ell_2^2 = \xx^T(W_1^TP_{12}W_2^T)P_{12}(W_2P_{12}W_1)\xx,\text{ and }
    \]
    \[
        p_{1,3}(\xx)=(c_1b_{22}+c_2b_{12})\ell_1+(c_1b_{21}+c_2b_{11})\ell_2=W_3P_{12}W_2P_{12}W_1\xx.
    \]
\end{example}

\begin{lemma}
\label{lemma:binary-degrees}
    Let $d=(2,2,\dots,2,d_L)$ with $d_L\geq 1$. Let $P_{i,\ww}$ and $Q_{\ww}$ be defined in~\eqref{eq:RNN-coordinates-formulas-binary}. Then the degrees of $P_{i,\ww}$ and $Q_{\ww}$ are given by
    \begin{equation}
    \label{eq:coordinate-output-degree-of-BRNN}
        \begin{aligned}
            n(\dd)=\deg(P_{i,\ww}(\xx)) &= L + \delta(L)-1,\\
            m(\dd)=\deg(Q_{\ww}(\xx)) &= L-\delta(L),
        \end{aligned}
    \end{equation}
    where $\delta(L)=0$ of $L$ is even $\delta(L)=1$ if $L$ is odd.
\end{lemma}
\begin{proof}
    See Appendix~\ref{apx:binary-degrees}
\end{proof}

\subsection{Filling architectures}
\label{sec:5-filling-archtiectures}

We want to find all possible architectures for which there are enough parameters to potentially fill the entire ambient space.

\begin{lemma}
\label{lemma:binary-params-dim}
    Let $N(L,d_L)$ be the number of parameters and $M(L,d_L)$ be the ambient space dimension. If $L>2$, then
    \[
        N(L,d_L)\ge M(L,d_L)\quad\Longleftrightarrow\quad
        1\le d_L\le 3+\frac{1-2\delta(L)}{L+\delta(L)-2},
    \]
    where $\delta(L)=0$ if $L$ is even and $\delta(L)=1$ if $L$ is odd. In particular, for $L$ even this is equivalent to $d_L\in\{1,2,3\}$, and for $L$ odd to $d_L\in\{1,2\}$.
\end{lemma}
\begin{proof}
    See Appendix~\ref{apx:binary-params-dim}.
\end{proof}
Therefore, the only possible binary deep neural network architectures that can be filling are of the form $\dd=(2,\dots,2,d_L)$ for $L>2$ and  $d_L\leq 3$. 

\begin{remark}
    The case $L=2$ was covered in Proposition \ref{prop:neuromanifold-(2,m,k)} where the neurovariety $\mathcal{V}_{\dd,\sigma}$ with architecture $(2,2,k)$ is filling for all $k\geq 1$.
\end{remark}

As the following Proposition shows, only $d_L=1$ produces a filling architecture for deep binary neural networks. 

\begin{proposition}
\label{prop:binary-2..21-architecture}
    If $\dd=(2,\dots,2,1)$ and $\sigma(x) = 1/x$, then the neurovariety $\mathcal{V}_{\dd,\sigma}(\C)$ is filling, i.e.,
    \[
        \mathcal{V}_{\dd,\sigma}(\C)=S^{n(\dd)}(\C^2)\times S^{m(\dd)}(\C^2).
    \]
\end{proposition}
\begin{proof}
    See Appendix~\ref{apx:binary-2..21-architecture}.
\end{proof}

If we have more than one output, then the polynomials $P_{i,\ww}$ differ only by one root according to equations~\eqref{eq:RNN-coordinates-formulas-binary}. This shows the following. 

\begin{proposition}
\label{prop:binary-not-dense-arch}
    If $\dd = (2,\dots,2,k)$ and depth $k>1$, then $\mathcal{V}_{\dd,\sigma}(\C)$ is not filling, i.e.,
    \[
        \mathcal{V}_{\dd,\sigma}(\C)\subsetneq (S^{n(\dd)}(\C^2))^k\times S^{m(\dd)}(\C^2).
    \]
\end{proposition}
\begin{proof}
    See Appendix~\ref{apx:binary-not-dense-arch}.
\end{proof}

\section{Numerical Experiments}
\label{sec:numerics}

In this section, we present numerical experiments with rational neural networks. Section~\ref{sec:6-dimensions} introduces our main conjecture about the dimension of neurovarieties and includes numerical verification for a range of architectures. In Section~\ref{sec:6-training}, we illustrate how such networks can be trained to learn the locations of singularities of a meromorphic function.

\subsection{Dimensions}
\label{sec:6-dimensions}

To compute the dimension of the neuromanifold, we follow the methods introduced in~\cite{KILEEL19}. Specifically, we compute the rank of the Jacobian of the combinatorial parameter map $\Psi_{\dd,\sigma}$ defined in~\eqref{eq:parameter-map-combinatorial-general} numerically over a finite field with a sufficiently larger prime number using \emph{SAGE} package. To obtain an upper bound on the dimension of the neurovariety, we examine the dimension of a generic fiber of the combinatorial parameter map.

\begin{lemma}
\label{lemma:dim-generic-fiber}
    The dimension of the generic fiber of the combinatorial parameter map $\Psi_{\dd,\sigma}$ defined in~\eqref{eq:parameter-map-combinatorial-general} is at least
    \[
        \sum_{i=1}^{L-1}d_i-1.
    \]
\end{lemma}
\begin{proof}
    See Appendix~\ref{apx:dim-generic-fiber}
\end{proof}

If the neurovariety is not filling, then its generic dimension is bounded by
\[
    \dim(\mathcal{V}_{\dd,\sigma})
    \;\leq\; \sum_{i=1}^{L} d_i d_{i-1} \;-\; \sum_{i=1}^{L-1} d_i \;+\; 1.
\]
We compute the dimensions of the neurovariety for all possible architectures whose number of parameters $N$ is bounded by $30$ and whose number of layers $L$ is bounded by $5$, which results in architectures $722$ to check. The total computation time was $9$ hours. For each architecture, we set a timeout of $10$ seconds to accelerate the process. Table~\ref{table:dimension-results} provides several examples of the architectures we computed.

\begin{table}[h]
\centering
\caption{Dimensions of the neurovariety, 
ambient space, and parameter count.}
\label{table:dimension-results}
\begin{tabular}{cccccc}
\hline\hline
$\dd$ & $\dim \mathcal{V}_{\dd,\sigma}$ & Ambient dim.
 & Param. count & Runtime [s] & Conjectured $\dim \mathcal{V}_{\dd,\sigma}$ \\
\hline
    $[3, 3, 3, 3]$ & $22$  & $136$  & $27$  & $0.759$ & $22$  \\
    $[2, 3, 4, 3]$ & $24$  & $39$  & $30$ & $0.659$ & $24$  \\
    $[4, 3, 2, 2, 3]$ & $22$  & $372$  & $28$  & $4.116$ & $22$  \\
    $[2, 2, 2, 3, 2, 1]$ & $14$ & $15$ & $22$ & $0.649$ & $14$ \\
    $[2, 2, 4, 2, 2, 1]$ & $17$ & $23$ & $26$ & $9.620$ & $17$ \\
\hline\hline
\end{tabular}
\end{table}

Based on our computations, we confirmed that the computed dimensions indeed agree with the predicted ones. This leads us to the following conjecture.
\begin{conjecture}
    For $\dd=(d_0,d_1,\dots,d_L)$ and $\sigma(x)=1/x$ the dimension of the neurovariety is 
    \[
        \dim(\mathcal{V}_{\dd,\sigma})
        =\max\left(
            \sum_{i=1}^L d_i d_{i-1} - \sum_{i=1}^{L-1} d_i + 1,\;\;
            d_L \binom{d_0 + n(\dd)-1}{n(\dd)} + \binom{d_0 + m(\dd)-1}{m(\dd)}
        \right).
    \]
\end{conjecture}
We have already proven this conjecture for  architectures $\dd=(n,2,m)$ in Theorem~\ref{theorem:n2m-architecture-ideal}. We also found the architecture $\dd=(2,3,2,1)$ to be particularly interesting, as the neurovariety in this  case is an irreducible hypersurface.

\begin{example}
    The architecture $\dd=(2,3,2,1)$ is not filling, as $\dim(\mathcal{M}_{\dd,\sigma})=10$ while the dimension of the ambient space equals $11$. Using the \texttt{MultigradedImplicitization} package for Macaulay2 by Cummings and Hollering \cite{CUMMINGS23}, we verified that the defining equation must have degree at least 5. Unfortunately, we were unable to compute  the  equation itself as the computation did not finish after $10$ hours of running.
\end{example}

\begin{problem}
    Determine the defining equation of the neurovariety corresponding to the architecture $\dd=(2,3,2,1)$. Moreover, identify all other architectures whose associated neurovarieties are hypersurfaces.
\end{problem}

\subsection{Training}
\label{sec:6-training}

Training neural networks with discontinuous activation functions is not uncommon. For example, JumpReLU is used in \cite{ERICHSON19} and provides an improvement against adversarial attacks. Moreover, training shallow networks is closely related to determining the location of linear poles of a meromorphic function in several variables \cite{DAHMEN24}. In this setting, the rows of $W_1$ correspond to the location of the poles.

As a proof of concept, we consider the meromorphic function
\[
    g(x,y)=\frac{1}{x+y} + \frac{1}{x-y},
\] 
defined on $[-1,1]^2$ with singularities excluded. We uniformly sample a lattice of size $21 \times 21$ on this square, see Figure~\ref{fig:training-data}.

\begin{figure}[ht]
    \centering
    \begin{minipage}[t]{0.48\textwidth}
        \centering
        \includegraphics[width=0.6\textwidth]{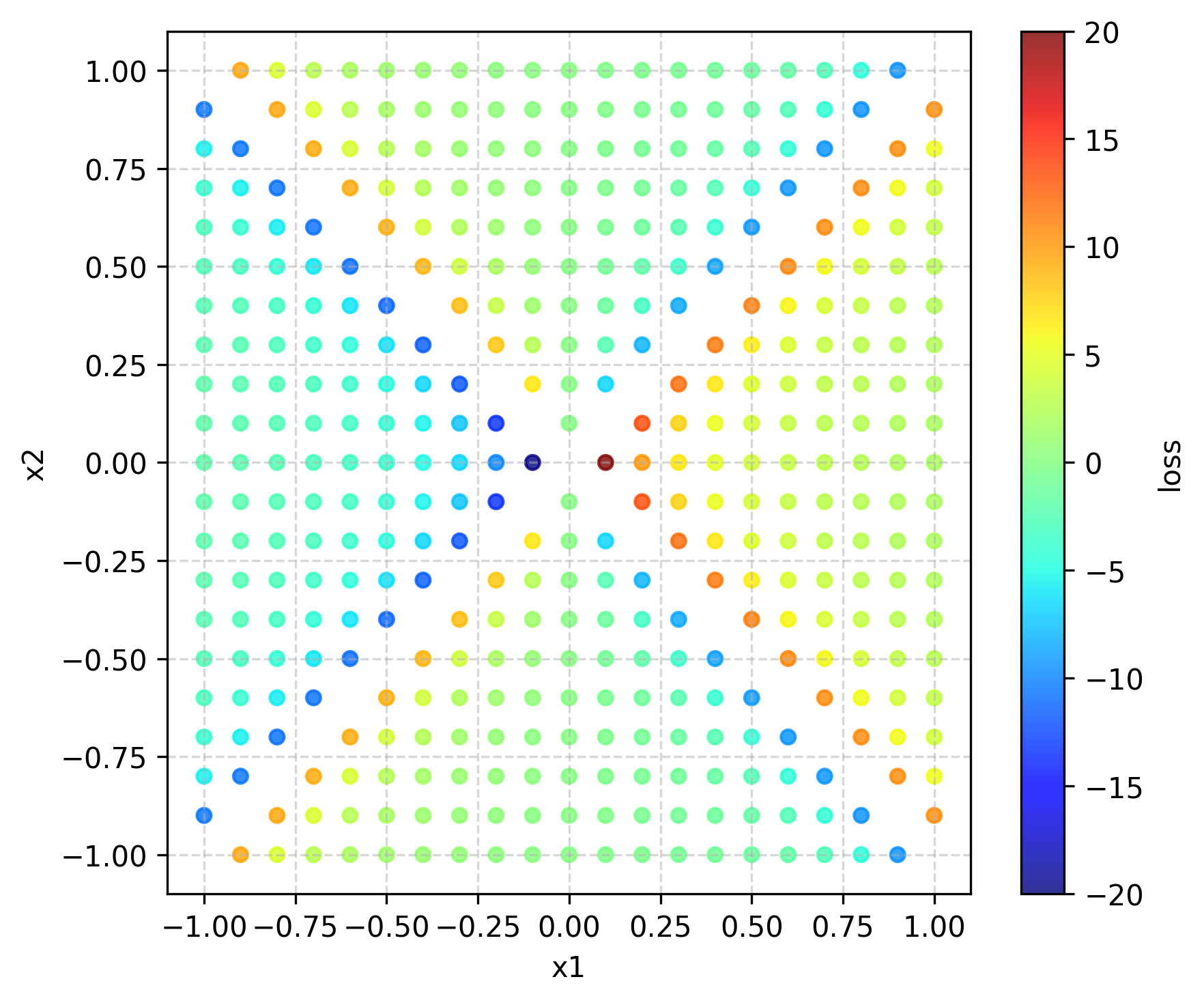}
        \caption{Training data}
        \label{fig:training-data}
    \end{minipage}%
    \hfill
    \begin{minipage}[t]{0.48\textwidth}
        \centering
        \includegraphics[width=0.6\textwidth]{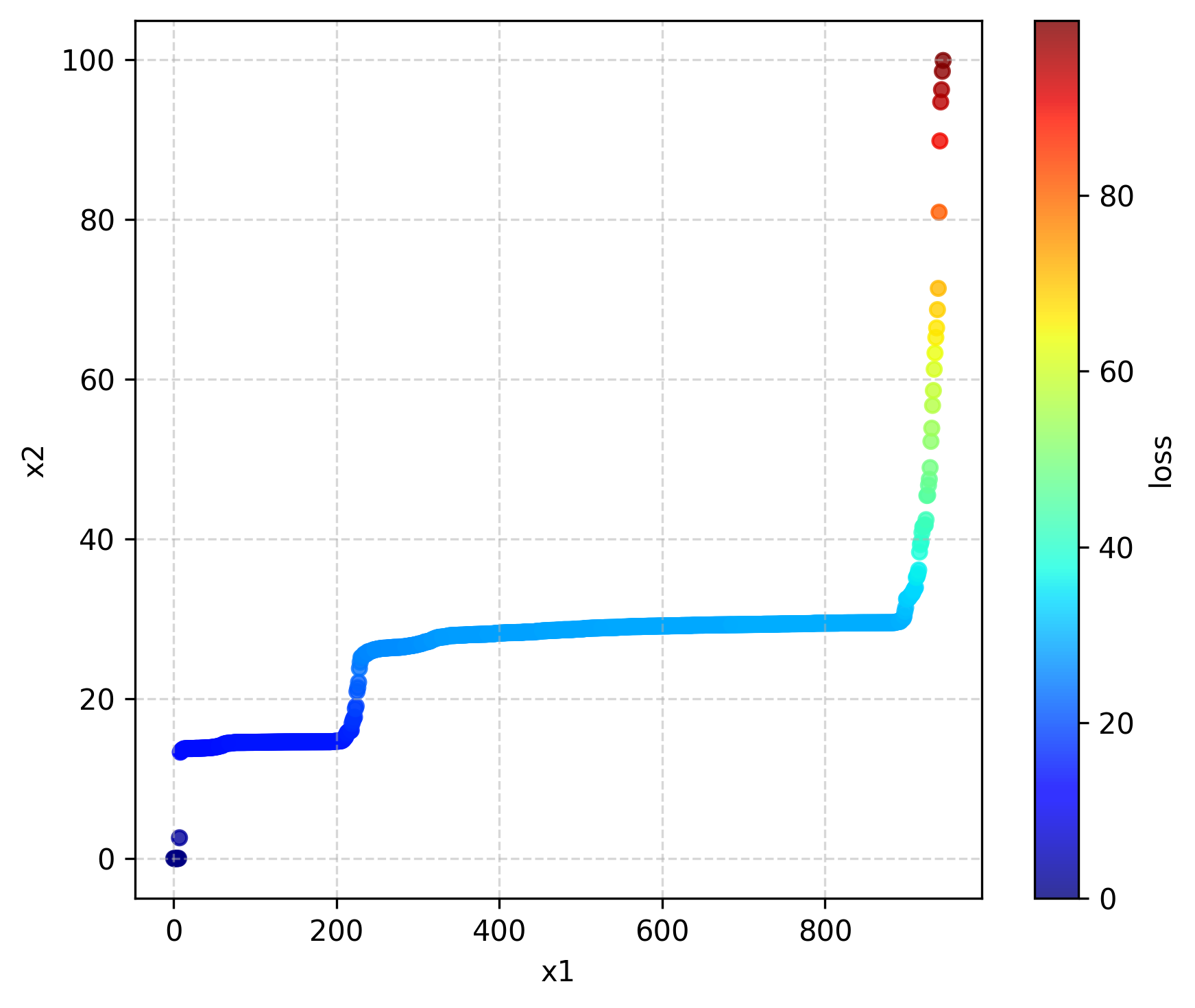}
        \caption{Final loss}
        \label{fig:final-loss}
    \end{minipage}
\end{figure}

We train a rational neural network with architecture $\dd = (2,2,1)$ to approximate $g$. To understand the learning dynamics, we perform $1000$ random initializations of the weights using the Xavier initialization scheme~\cite{GLOROT10}. Training is carried out for $20{,}000$ epochs using the Adam optimizer (learning rate $10^{-3}$) implemented in Torch. The loss function is the mean squared error, and full-batch training is applied.

We observed that the success rate of the loss converging to zero was only about $1\%$, while the success rate of learning at least one singularity was around $20\%$; see Figure~\ref{fig:final-loss}. Although the probability of learning both singularities simultaneously was small, the network was able to capture at least one singularity with considerably higher success.

\begin{figure}[ht]
    \centering
    \begin{subfigure}[t]{0.48\textwidth}
        \centering
        \begin{subfigure}[t]{0.48\textwidth}
            \centering
            \includegraphics[width=\linewidth]{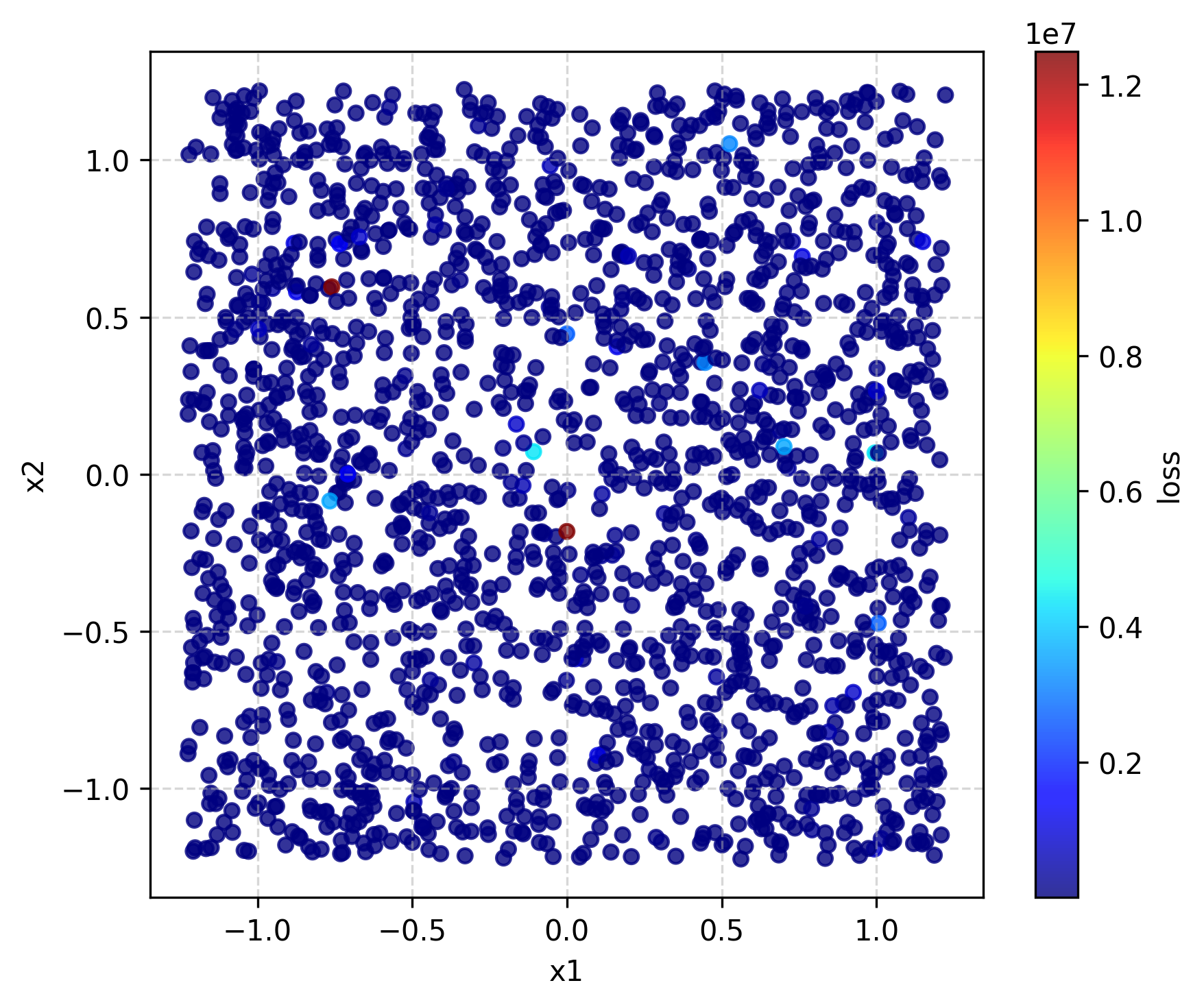}
            \caption{Rows of $W_1$'s before training}
        \end{subfigure}
        \hfill
        \begin{subfigure}[t]{0.48\textwidth}
            \centering
            \includegraphics[width=\linewidth]{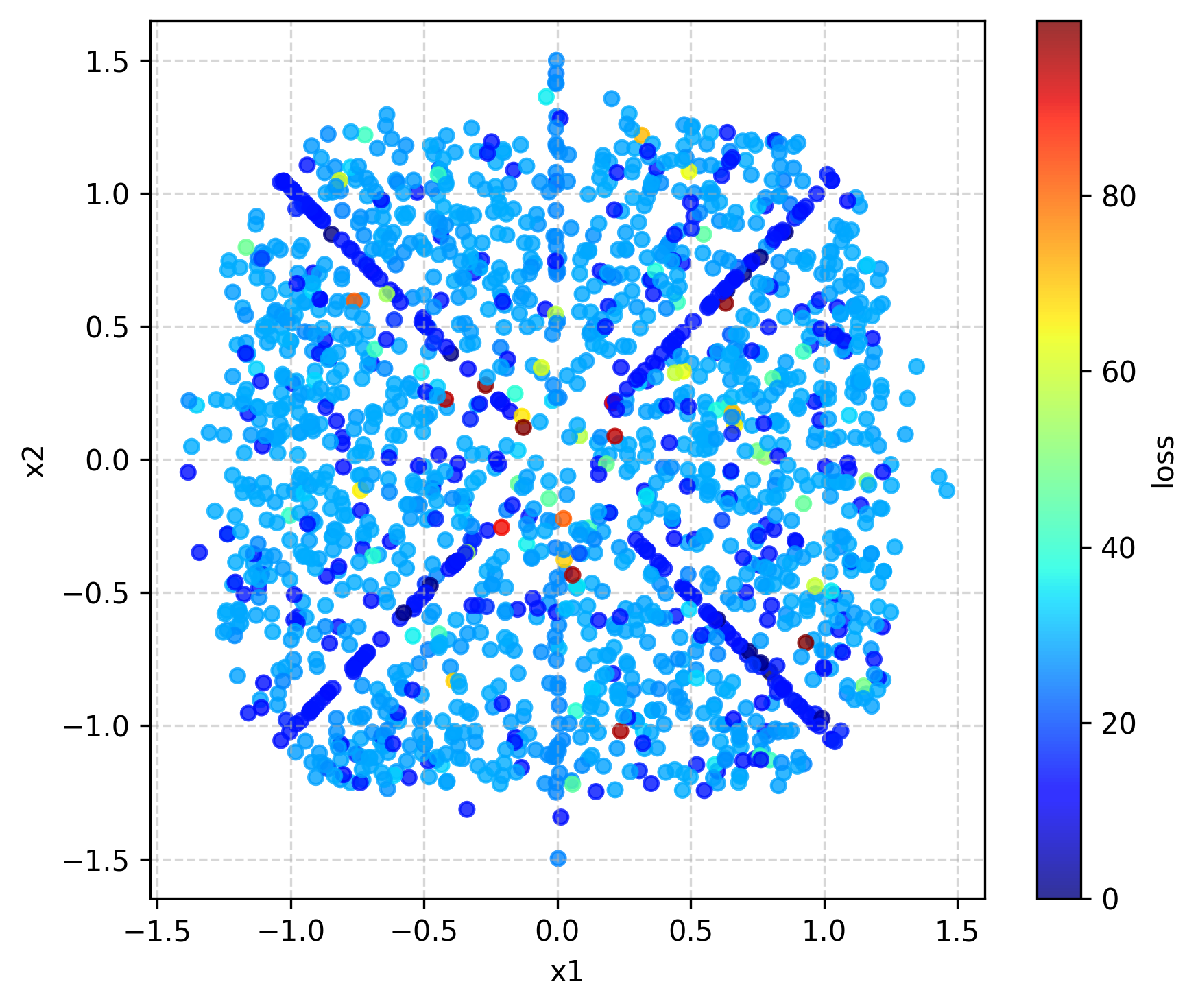}
            \caption{Rows of $W_1$'s after training}
        \end{subfigure}
    \end{subfigure}
    \hfill
    \begin{subfigure}[t]{0.48\textwidth}
        \centering
        \begin{subfigure}[t]{0.48\textwidth}
            \centering
            \includegraphics[width=\linewidth]{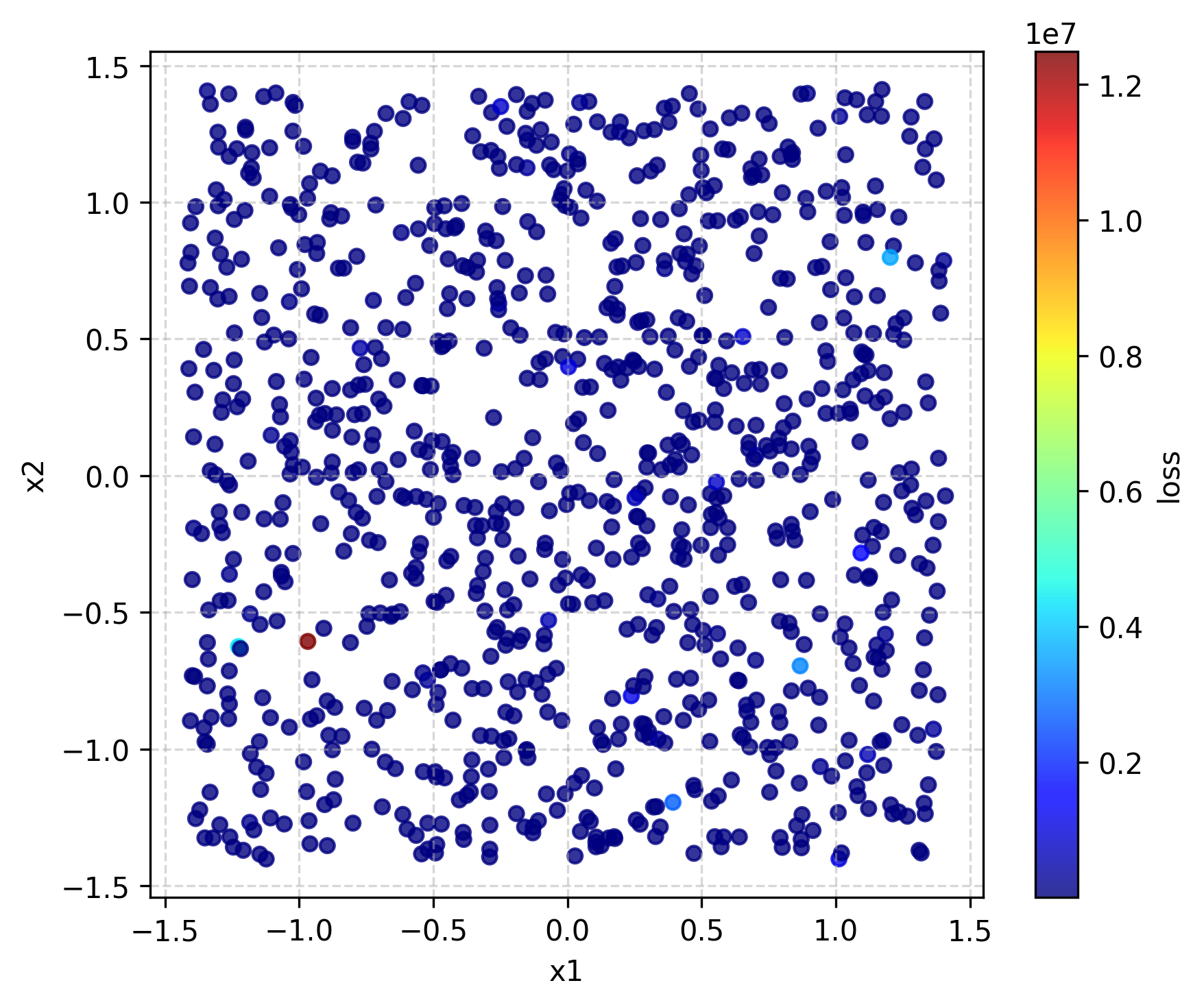}
            \caption{$W_2$'s before training}
        \end{subfigure}
        \hfill
        \begin{subfigure}[t]{0.48\textwidth}
            \centering
            \includegraphics[width=\linewidth]{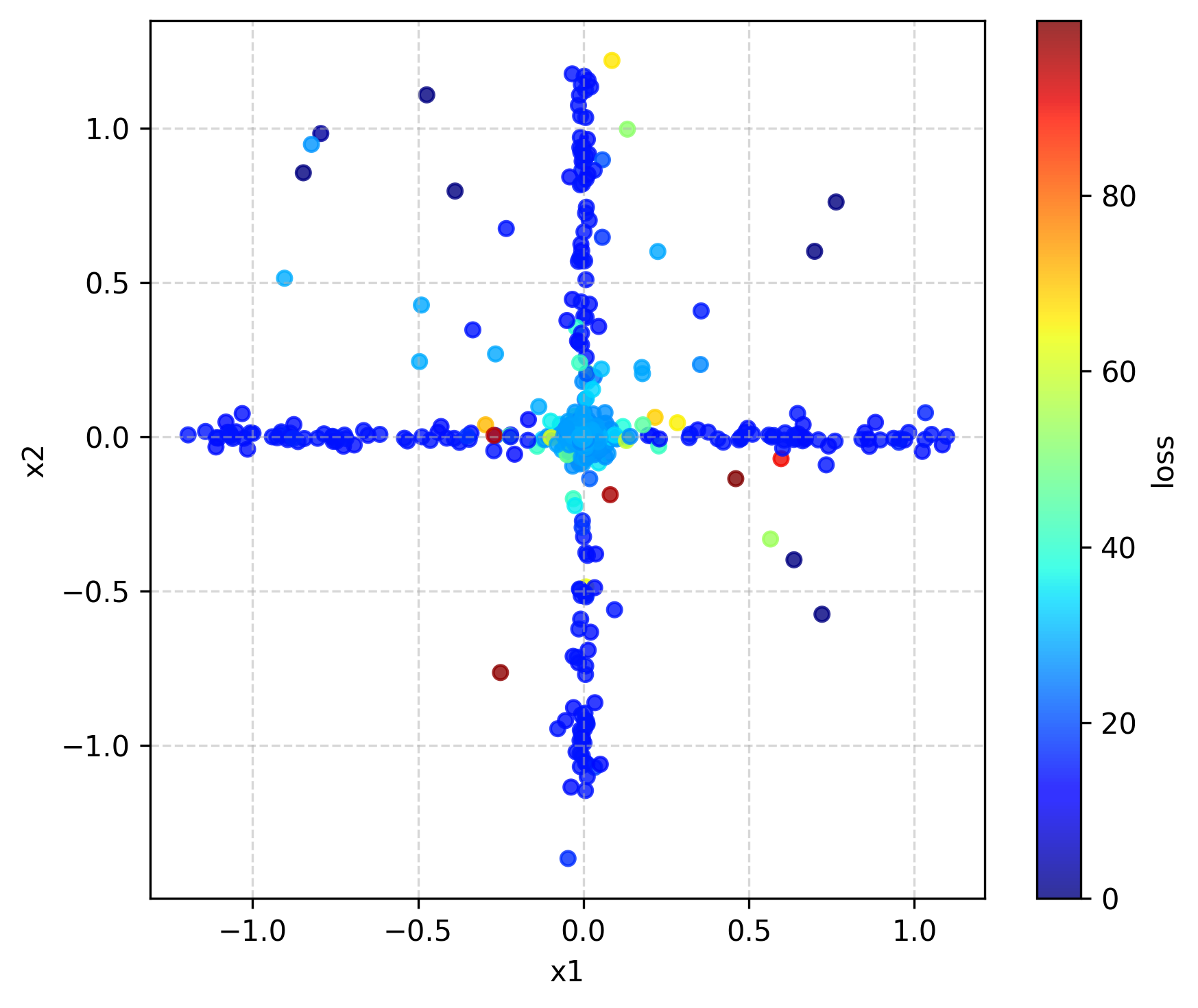}
            \caption{$W_2$'s after training}
        \end{subfigure}
    \end{subfigure}

    \caption{The dynamic of parameters changes during training}
    \label{fig:w1-w2-dynamic}
\end{figure}

Figure~\ref{fig:w1-w2-dynamic} depicts the rows of $W_1$ and $W_2$ before and after training. We observe that most failures of the network to learn the singularities are due to the vanishing of entries in $W_2$. Recall that the entries of $W_2$ serve as coefficients in the linear combination of the inverses of the linear forms.

Therefore, rational neural networks have the potential to learn the locations of singularities from data. In particular, the rows 
 of $W_1$ correspond to the singularities of the function. However, this is only the case if the 
 cost function 
 drops to 0, which is hard to achieve. This shows that 
 there is a trade-off between 
 the interpretability of the weights and the stability of the optimization 
 procedure.
 More advanced optimization techniques may be able to significantly improve the performance of the network. 

\section{Conlcusion ans Future Work}
\label{sec:conclusion}

In this paper, we introduced and studied the neuromanifolds and the neurovarieties associated with rational neural networks via a combinatorial parameter map. 
For shallow neural networks we proposed two algorithms of algebraic nature that allow us to recover the parameters of a given rational function in the image of the network. 
Further, we found polynomial equations in the ideals of the variety,
and even described the full ideal for some models. 
We analyzed possible filling architectures for shallow and binary deep architectures. 
We also conjectured a general formula for the dimension of neurovarieties and demonstrated that rational neural networks can be trained in practice.

Our line of work opens numerous possibilities for future research
in the area of rational neural networks.
Possible  directions include studying networks with more general activation functions of the form $p/q$, rather the special case $\sigma(x)=1/x$. Another important direction is to further develop the description of neurovarieties, analyze their singularities, and explore the connections between the geometry of neuromanifolds/neurovarieties and the training dynamics of rational neural networks. 

\subsection*{Acknowledgments}
This research was initiated while the authors were visiting the Algebraic Statistics and Our Changing World program at the Institute for Mathematical and Statistical Innovation (IMSI), supported by the National Science Foundation (Grant No. DMS-1929348). Elina Robeva and Maksym  Zubkov were supported by a Canada CIFAR  AI Chair award and an NSERC Discovery Grant (DGECR-2020-00338).

\bibliographystyle{plain}
\bibliography{references}

\appendix
\section{Proofs}

\subsection{Proof of Theorem~\ref{theorem:neural-network-closed-form-general}}
\label{apx:neural-network-closed-form-general}
\begin{proof}
    We proceed by induction on the number of layers $L$. For $k\geq 1$, define the index sets
    \[
        S_k\coloneqq\{k-1,k-3,\dots,\delta(k+1)\},\qquad T_k\coloneqq\{k,k-2,\dots,\delta(k)\}.
    \]
    We claim that after $k$ layers, the network output can be written as
    \begin{equation}
    \label{eq:main-theorem-k-step}
        f_k(\xx)=
        \begin{bmatrix}
            P_{1,k}(\xx)/Q_k(\xx)\\ 
            \vdots\\ 
            P_{d_k,k}(\xx)/Q_k(\xx)
        \end{bmatrix},
        \quad
        P_{i,k}(\xx)=p^{(k)}_i(\xx)\!\!\!\prod_{t\in S_k}\!\! q^{(t)}(\xx),
        \quad
        Q_k(\xx)=\!\!\!\prod_{t\in T_k}\!\! q^{(t)}(\xx),
    \end{equation}
    for $i=1,\dots,d_k$.
    
    \noindent
    \emph{Base case $k=1$.} We have $f_1(\xx)=W_1\xx=\begin{bmatrix}
        \ell_1\dots\ell_{d_1}
    \end{bmatrix}^T$. Here $S_1=\{0\}$ and $T_1=\{1\}$, giving $P_{i,1}=p^{(1)}_i q^{(0)}$ and $Q_1=q^{(1)}$ with $p^{(1)}_i=\ell_i$, and $q^{(0)}=q^{(1)}=1$ that matches~\eqref{eq:main-theorem-k-step}.

    \emph{Inductive step.} Assume \eqref{eq:main-theorem-k-step} holds for some $k\geq1$. We must show that it holds for $k+1$. Let us compute $W_{k+1}\sigma(f_k)$. Applying $\sigma$ entrywise, we will obtain
    \[
        \sigma(f_k)(\xx) = 
        \begin{bmatrix}
            Q_k/P_{1,k}\\ 
            \vdots\\ 
            Q_k/P_{d_k,k}
        \end{bmatrix} = \frac{1}{q_k}
        \begin{bmatrix}
            1/p^{(k)}_1\\ 
            \vdots\\ 
            1/p^{(k)}_{d_k}
        \end{bmatrix},
        \quad
        q_k\coloneqq\frac{\prod_{t\in S_k} q^{(t)}}{\prod_{t\in T_k} q^{(t)}}
        =
        \frac{q^{(k-1)}q^{(k-3)}\cdots q^{(\delta(k+1))}}{q^{(k)}q^{(k-2)}\cdots q^{(\delta(k))}}.
    \]
    Multiplying by $W_{k+1}\in\R^{d_{k+1}\times d_k}$ yields
    \[
        W_{k+1}\sigma(f_k)
        =\frac{1}{q_k}
        \begin{bmatrix}
            \sum_{i=1}^{d_k} w_{k+1,1,i}\,(1/p^{(k)}_i)\\
            \vdots\\
            \sum_{i=1}^{d_k} w_{k+1,d_{k+1},i}\,(1/p^{(k)}_i)
        \end{bmatrix}
        =\frac{1}{q_k\prod_{j=1}^{d_k}p^{(k)}_j}
        \begin{bmatrix}
            \sum_{i=1}^{d_k} w_{k+1,1,i}\,\widehat{p}^{(k)}_{i}\\
            \vdots\\
            \sum_{i=1}^{d_k} w_{k+1,d_{k+1},i}\,\widehat{p}^{(k)}_{i}
        \end{bmatrix},
    \]
    where $\widehat{p}^{(k)}_{i}\coloneqq\prod_{s\neq i} p^{(k)}_s$. Since $q^{(k+1)}=\prod_{j=1}^{d_k}p^{(k)}_j$ by recursive definition and $\delta(k)=\delta(k+2)$, then
    \[
        W_{k+1}\sigma(f_k)
        =\frac{1}{q_k\,q^{(k+1)}}
        \begin{bmatrix}
            p^{(k+1)}_1\\ 
            \vdots\\ 
            p^{(k+1)}_{d_{k+1}}
        \end{bmatrix}
        =\frac{q^{(k)}q^{(k-2)}\cdots q^{(\delta(k+2))}}{q^{(k+1)}q^{(k-1)}q^{(k-3)}\cdots q^{(\delta(k+1))}}
        \begin{bmatrix}
            p^{(k+1)}_1\\ 
            \vdots\\ 
            p^{(k+1)}_{d_{k+1}}
        \end{bmatrix}.
    \]
    This matches \eqref{eq:main-theorem-k-step} for $k+1$, completing the induction.
\end{proof}

\subsection{Proof of Lemma~\ref{lemma:degrees-for-p-and-q}}
\label{apx:degrees-for-p-and-q}
\begin{proof}
    Let $n_k\coloneqq\deg(p^{(k)}_{i})$ and $m_k\coloneqq\deg(q^{(k)})$. From the recursion relations~\eqref{eq:main-recursion-formula} and $\deg(fg)=\deg(f)+\deg(g)$ and $\deg(f+g)=\max\{\deg(f),\deg(g)\}$, one checks that the degrees satisfy the following recursive relation for $k\geq 2$
    \begin{align*}
        n_{k} &= (d_{k-1}-1)n_{k-1},\\
        m_{k} &= d_{k-1}n_{k-1}
    \end{align*}
    with initial conditions $n_1=1$ and $m_1=0$. The desired formulas for the degrees of $p_{i}^{(k)}$ and $q^{(k)}$ are obtained by recursively expressing $n_i$ in terms of $n_{i-1}$ until reaching the initial condition $n_1$.  
\end{proof}

\subsection{Proof of Lemma~\ref{lemma:total-degree}}
\label{apx:total-degree}
\begin{proof}
    The result follows from applying Lemma~\ref{lemma:degrees-for-p-and-q} to the definitions of $P_{i,\ww}$ and $Q_{\ww}$ in \eqref{eq:RNN-coordinates-formulas}, while taking into account the alternating structure of the indices.
\end{proof}

\subsection{Proof of Lemma~\ref{lemma:fiber-symmetries-general}}
\label{apx:fiber-symmetries-general}
\begin{proof}
    Let $\ww^{\prime}$ be the transformed weights $\ww$. It suffices to show that for each $i=1,\dots,d_L$, the $i$th coordinate of the output is unchanged under the transformation, that is, $f_{i,\ww^{\prime}}(\xx)=f_{i,\ww}(\xx)$ for all $\xx$ in the definition domain. 
    
    Since $\sigma$ acts coordinate-wise, multiplying by a permutation matrix on the left permutes coordinates before applying $\sigma$, while multiplying on the right by the inverse permutation after applying $\sigma$ undoes the original reordering.
    
    Next, we check that the output is also invariant under the action of diagonal matrices. Fix a layer index $k\in\{1,\dots,L-1\}$. We want to show that
    \[
        W_{k+1}D_{k}\sigma(D_{k}W_{k}f^{(k-1)}_{\ww})=W_{k+1}\sigma(W_kf^{(k-1)}_{\ww})
    \]
    where $f^{(k)}_{\ww}$ is the output after $k$ layers. Let $l_{i,k}\coloneqq(W_{k}f^{(k-1)}_{\ww})_i$ be the $i$th coordinate of $W_kf^{(k-1)}_{\ww}$, and let $\lambda_{i,k}$ be the $i$th diagonal entry of $D_k\in\R^{d_k\times d_k}$. Then
    \[
        D_{k}W_{k}f^{(k-1)}_{\ww} = \begin{bmatrix}
            \lambda_{1,k}l_1 & \lambda_{2,k}l_2 & \dots \lambda_{d_{k},k}l_{d_{k}}
        \end{bmatrix}^T.
    \]
    Applying $W_{k+1}D_k\sigma$ gives
    \[
        ( W_{k+1}D_{k}\sigma(D_{k}W_{k}f^{(k-1)}_{\ww}))_i = \left(W_{k+1}D_{k}\begin{bmatrix}
            1/(\lambda_{1,k}l_1) & \dots & 1/(\lambda_{1,d_{k}}l_{d_{k}})
        \end{bmatrix}\right)_i =
    \]
    \[
       = \frac{\lambda_{1,k}w_{k+1,i,1}}{\lambda_{1,k}\ell_1}+\dots+\frac{\lambda_{1,d_k}w_{k+1,i,d_k}}{\lambda_{1,d_k}\ell_{d_k}} = \frac{w_{k+1,i,1}}{\ell_1}+\dots+\frac{w_{k+1,i,d_k}}{\ell_{d_k}} = (W_{k+1}\sigma(W_kf^{(k-1)}_{\ww}))_i,
    \]
    for all $i$, which proves the claim.
\end{proof}

\subsection{Proof of Lemma~\ref{lemma:closed-tensor-form-shallow-network}}
\label{apx:closed-tensor-form-shallow-network}
\begin{proof}
    The equations~\ref{eq:shallow-network-equations} follow from the identity
    \[
        x_1 x_2 \cdots x_m \;=\; \Sym(e_1\otimes e_2\otimes \cdots \otimes e_m)\circ \xx^{\otimes m},
        \quad
        \xx=\begin{bmatrix}x_1 & x_2 & \cdots & x_m\end{bmatrix}^{\!T}.
    \]
    Indeed, if $\ell_i=(W_1\xx)_i$ for $i=1,\dots,m$, then setting $y:=W_1\xx$ gives
    \[
        \ell_1\ell_2\cdots \ell_m \;=\; \prod_{i=1}^m y_i
        \;=\; \Sym(e_1\otimes e_2\otimes \cdots \otimes e_m)\circ y^{\otimes m}
        \;=\; \Sym(e_1\otimes \cdots \otimes e_m)\circ (W_1\xx)^{\otimes m}.
    \]
    This is exactly the desired form. The second identity is obtained in the same way.
\end{proof}

\subsection{Proof of Lemma~\ref{lemma:shallow-network-params-vs-amb-dim}}
\label{apx:shallow-network-params-vs-amb-dim}
\begin{proof}
    Fix $m\geq1$ and $k\geq1$.  We split the analysis by the input dimension~$n$.

    \noindent
    \textbf{Case $\boldsymbol{n=1}$.}
    Here $M(1,m,k)=1+k$ while $N(1,m,k)=m(1+k)\geq1+k$ as $m\geq 1$.  Hence $N\geq M$ for all $(m,k)$.

    \noindent
    \textbf{Case $\boldsymbol{n=2}$.}
    We have
    \[
         M(2,m,k)=\binom{m+1}{m}+k\binom{m}{m-1}=m+1+km
         \quad\text{and}\quad
         N(2,m,k)=2m+km.
    \]
    Since $2m+km\geq m+1+km$ for every $m\geq 1$, the inequality
    $N\geq M$ again holds for all $(m,k)$.
    
    \noindent
    \textbf{Case $\boldsymbol{n\geq3}$.}
    For $n\geq3$ the first binomial term :
    \[
         \binom{n+m-1}{m}\geq mn
    \]
    for a fixed $n$ for all $m$. For the second binomial coefficient, we have that  
    \[
        \binom{n+m-2}{m-1} \geq \frac{m(m+1)}{2} > m,
    \]
    so $N(n,m,k)<M(n,m,k)$. Therefore the inequality $N\geq M$ fails for all $n\geq3$.
    
    \medskip
    Combining the three cases, we conclude that $N(n,m,k)\geq M(n,m,k)$ holds if and only if $n\in\{1,2\}$ with $m\geq1$ and $k\geq1$.
\end{proof}

\subsection{Proof of Lemma~\ref{lemma:lin-independence-of-g_hat_i}}
\label{apx:lin-independence-of-g_hat_i}
\begin{proof}
    Suppose that $\{\hat{\ell}_1,\hat{\ell}_2,\dots,\hat{\ell}_m\}$ are linearly dependent. Without loss of generality, assume the coefficient of $\hat{\ell}_m$ is non-zero. Then $\hat{\ell}_m$ can be expressed as a linear combination of $\hat{\ell}_1, \dots, \hat{\ell}_{m-1}$, i.e.,
    \[
        \ell_1 \ell_2 \dots \ell_{m-1} = \hat{\ell}_m = \sum_{i=1}^{m-1} a_i \hat{\ell}_i.
    \]
    Since each $\hat{\ell}_i$ is divisible by $\ell_m$ for all $i=1,\dots,m-1$, then $\ell_1\ell_2\dots \ell_{m-1}=\ell_m p$ for some polynomial $p \in S^{m-1}(\mathbb{C}^n)$. This implies that $\ell_m$ divides $\ell_i$ for some $i$, hence $\{\ell_i,\ell_m\}$ are linearly dependent.
\end{proof}

\subsection{Proof of Proposition~\ref{prop:neuromanifold-(2,m,k)}}
\label{apx:neuromanifold-(2,m,k)}
\begin{proof}
    First, let us show that $\mathcal{M}_{\dd,\sigma}(\C)$ is not filling. We want to find a point 
    \[
        (P_1,\dots,P_k,Q)\in\left(S^{m-1}(\mathbb{\C}^2)\right)^k \times S^m(\mathbb{\C}^2)    
    \]
    for which there are no parameters $\ww$ solving Equations~\eqref{eq:nmk-definining-system}. 
    
    Take $\ell_1(\xx)=\ell_2(\xx)=x_1$, and let $\ell_3,\dots,\ell_m$ be any nonzero linear forms. Then the polynomials $P_{i,\ww}$ and $Q_{\ww}$ are both divisible by $x_1$ for all $i$. Thus, take $P_i(\xx)=x_2^{m-1}$ for all $i$ and $Q(\xx)=\ell_1\ell_2\dots \ell_m$, then quations~\eqref{eq:nmk-definining-system} have no solutions $\ww$ as $P_i$ are not divisible by $x_1$. Hence, the neuromanifold is not filling.
    
    Next, we show that the neurovariety is filling. Consider the Zariski open set where the discriminant of $Q$ does not vanish 
    \begin{equation}
    \label{eq:2mk-Zariski-open}
        U\coloneqq (S^{m-1}(\C^2))^{k}\times S^{m}(\C^2)\setminus ((S^{m-1}(\C^2))^k\times V_Q),
    \end{equation}
    where $V_Q$ is the discriminant hypersurface in $S^m(\C^2)$. For any point $(P_1,\dots,P_k,Q)\in U$,
    the binary form $Q$ factors as a product of $m$ distinct linear forms 
    \[
        Q = \ell_1\ell_2\dots \ell_m.
    \]
    These linear forms $\ell_j$ determine the rows of $W_1$.
    
    To recover $W_2$, we need to solve the system of linear equations in \ref{eq:nmk-definining-system}. Since $S^{m-1}(\C^2)$ is an $m$-dimensional vector space and $\{\hat{\ell}_1,\dots,\hat{\ell}_m\}$ are linearly independent according to Lemma~\ref{lemma:lin-independence-of-g_hat_i}, then $\{\hat{\ell}_1,\dots,\hat{\ell}_m\}$ is a basis for $S^{m-1}(\C^2)$. Therefore, we can reconstruct the rows of $W_2$ uniquely knowing $W_1$ from the equations 
    \[
        P_{i} = b_{i1} \hat{\ell}_{1,m}+\dots+b_{im} \hat{\ell}_{m,m}.
    \]
    Thus every point in $U$ lies in the neuromanifold $\mathcal{M}_{\dd,\sigma}(\C)$. Because $U$ is Zarisky open, its closure is the enitre ambient space. In particular, the neurovariety $\mathcal{V}_{\dd,\sigma}(\C)$ is filling. 
\end{proof}


\subsection{Proof of Theorem~\ref{theorem:n2m-architecture-ideal}}
\label{apx:n2m-architecture-ideal}
\begin{proof}
    By explicitly writing out the composition of functions that define the neural network,
    we obtain the parametrization
    \begin{align*}
        \phi: \mathbb{R}[C] &\to  \mathbb{R}[a,b] \\
        C_{\e_i+\e_j} &\mapsto
  \begin{cases}
    a_{1i}a_{2j} & \text{for } i=j, \\
    a_{1i}a_{2j} + a_{1j}a_{2i}& \text{otherwise}, 
  \end{cases}\\
        C_{k,\e_i} &\mapsto b_{k1}a_{2i}+b_{k2}a_{1i}.
    \end{align*}
    Our goal is to show that the ideal $I:= \ker(\phi)$
    is generated by the $3\times 3$ minors of $\mathbf{M}$.

    From the above parametrization we deduce the following factorization of $\phi(\mathbf M)$:
    \[
    \phi(\mathbf{M}) = \begin{bmatrix}
            a_{21} & a_{11}\\
            a_{22} & a_{12}\\
            \vdots & \vdots \\
            a_{2n} & a_{1n}
        \end{bmatrix}\begin{bmatrix}
            b_{11} & b_{21} & \dots & b_{m1} & a_{11} & a_{12} & \dots &a_{1n} \\
            b_{12} & b_{22} &\dots & b_{m2}& a_{21} & a_{22} & \dots &a_{2n} 
        \end{bmatrix}.
    \]
    This is the product of the permuted transpose of $W_1$
    with $[W_2^{\top}|W_1]$.
    Crucially, this means that $\phi(\mathbf{M})$
    has rank at most $2$,
    so all its $3$-minors vanish, 
    implying that all $3$-minors of $\mathbf{M}$
    are in the ideal $I$. 

    For the opposite direction we use the fact that the ideal $\mathcal{M}_3(\mathbf M)$ generated
     by the $3$-minors of $M$ can be shown to be prime as a determinantal ideal of a partially symmetric generic matrix.
    Then equality follows from the fact that the two ideals in question have the same dimension.
    Indeed from the inclusion $\mathcal{M}_3(\mathbf M) \subseteq I$
    we obtain the series of (in)equalities 
    \begin{equation*}
        2(n+m)-1 \leq \dim I \leq \dim \mathcal{M}_3(\mathbf M)  = 2(n+m)-1.
    \end{equation*}
    The inclusion of ideals of the same dimension, 
    along with primality of $\mathcal{M}_3(\mathbf M)$ that we show in a separate lemma (Lemma~\ref{lemma:minors-ideal-prime}), show that the ideals are equal.
    In the remainder of this proof, we explain the first inequality and the last equality in the above equation. 

    We index the rows of the Jacobian of the map $\phi$ 
    by ordering first the variables $a$ and then $b$ lexicographically,
    and the columns first by using coefficients of the joint denominator 
    and then the numerators. 
    Then the Jacobian has an upper-triangular block form,
    so the rank is the sum of the ranks of the diagonal blocks. 
    Since we only want to bound the dimension from below, 
    it is enough to find values for the parameters that achieve the wanted rank. 
    We choose $a_{12} = a_{21}=1$ and set all other values to $0$.
    We examine the upper left block first, 
    corresponding to the variables $C_{\e_i+\e_j}$.
    The variable $a_{21}$ appears exactly in the $n$ columns corresponding to $C_{\e_1+\e_j}$,
    in the row of $a_{1i}$.
    The remainder of the entries in those columns are variables $a_{1i}$ that we set to $0$, 
    with the exception of $C_{\e_1+\e_2}$
    that has an additional $1$ in the row indexed by $a_{21}$.
    Similarly, $a_{12}$ appears exactly in the $n$ columns corresponding to $C_{\e_2+\e_j}$,
    in the row of $a_{21}$ and the rest are zeros.
    Thus, the rank of this block is $2n-1$,
    as we obtain different column unit vectors plus the column vector $C_{\e_1+\e_2}$ that has ones in the two remaining positions.
    
    Furthermore, for each $k = 1, \dots, m$
    the diagonal block corresponding to 
    columns $C_{k, \e_i}$
    and rows indexed by $b_{k1},b_{k2}$ 
    is a flipped version of $W_1$ and therefore contributes a rank of 2.
    Adding everything together we find the generic rank of this matrix to be at least $2n-1+2m = 2(n+m)-1$,
    which is a bound for the dimension of $I$.
    
    To find the dimension of $\mathcal{M}_3(\mathbf M)$ we use  Example 3.8 in \cite{CONCA94}.
    The matrix $\mathbf{M}$ is an $n \times (n+m)$ partially symmetric matrix 
    so the ideal of its minors of size $t =3 $ has dimension 
    \begin{equation*}
        (n+m+1-t/2)(t-1) = 2(n+m)-1
    \end{equation*}
    as we want to show. 
\end{proof}

\begin{example}
For $\dd = (3,2,2)$ the Jacobian takes the following form
\begin{equation*}
    \begin{bNiceArray}{cccccc|ccc|ccc}[first-row,first-col]
    & C_{2\e_1} &C_{\e_1+\e_2} &C_{2\e_2} &C_{\e_1+\e_3} &C_{\e_2+\e_3} &C_{2\e_3} &C_{1,\e_1} &C_{1,\e_2} &C_{1,\e_3} &C_{2,\e_1} &C_{2,\e_2} &C_{2,\e_3} \\
    a_{11} &a_{21} & a_{22} & 0 & a_{23} & 0 & 0 & \Block{6-3}<\Large>{*} & & & \Block{6-3}<\Large>{*} \\
    a_{12} &0 & a_{21} & a_{22} & 0 & a_{23} & 0 \\
    a_{13} &0 & 0 & 0 & a_{21} & a_{22} & a_{23} \\
    a_{21} &a_{11} & a_{12} & 0 & a_{13} & 0 & 0 \\
    a_{22} &0 & a_{11} & a_{12} & 0 & a_{13} & 0 \\
    a_{23} &0 & 0 & 0 & a_{11} & a_{12} & a_{13} \\
    \hline
    b_{11} &\Block{2-6}<\Large>{\mathbf{0}} & & & & & & a_{21} & a_{22}& a_{23} & \Block{2-3}<\Large>{\mathbf{0}} \\
     b_{12} & & & & & & &a_{11} & a_{12}& a_{13} & & & \\
     \hline 
     b_{21} &\Block{2-6}<\Large>{\mathbf{0}} & & & & & & \Block{2-3}<\Large>{\mathbf{0}} & & & a_{21} & a_{22}& a_{23}\\
     b_{22} & & & & & & & & & &  a_{11} & a_{12}& a_{13}
     
\end{bNiceArray}.
\end{equation*}
We set $a_{12} = a_{21} =1$
and the remaining values to $0$, 
as in the proof of Theorem \ref{theorem:n2m-architecture-ideal}.
The first five columns of the upper left block are linearly independent, and the last column is the zero vector, 
so the rank is five. 
Each of the remaining two diagonal blocks have rank $2$.
The rank of the Jacobian is thus $5+2+2 = 9 = 2(3+2)-1$.
\end{example}

\begin{lemma}
\label{lemma:minors-ideal-prime}
    Let $S$ be a general symmetric matrix of indeterminates $s_{ij}$, 
    $G$ a general matrix of indeterminates $g_{ij}$
    such that the two matrices have the same number of rows,
    and $d$ a positive integer.
    Then, the ideal of minors  
    \[\mathcal{M}_d([S|G])\]
    of the concatenation of the two matrices is prime. 
\end{lemma}
\begin{proof}
    We consider a generic symmetric matrix $T$ in new indeterminates $t_{ij}$, 
    so that the new block matrix
    \[M = \begin{bmatrix} 
    S & G \\ 
    G^{\top} & T 
    \end{bmatrix}\]
    is symmetric.
    Then we have 
    \[\mathcal{M}_d([S|G]) = \mathcal{M}_d(M) \cap \mathbb{R}[s_{ij},g_{ij}].\]
    Indeed, for the one direction 
    we observe that all minors of the smaller matrix $[S|G]$ are also minors of the larger matrix, 
    and all the monomials appearing involve only $s_{ij}$ and $g_{ij}$.
    For the other direction, we consider any lexicographic eliminating order for the variables $t_{ij}$, 
    i.e., an order such that any monomial involving at least one of the variables $t_{ij}$ is larger than any monomials not involving any $t_{ij}$.
    By \cite[Theorem 2.8]{CONCA94}, the $d$-minors of $M$ are a Gröbner basis for the corresponding ideal,
    so intersecting with $\mathbb{R}[s_{ij},g_{ij}]$
    gives a Gröbner basis for the elimination ideal.

    The larger ideal $ \mathcal{M}_d(M)$ is prime as the ideal of minors of a general symmetric matrix.
    Therefore, the elimination ideal $\mathcal{M}_d([S|G])$ is also prime.
\end{proof}

\subsection{Proof of Proposition~\ref{prop:minorideals}}
\label{apx:minorideals}
\begin{proof}
    In this proof we denote by $\{{\dot{l}_{d_1}} \}$
    the set $[d_1]\setminus \{l\}$,
    by $S_d(i_1,\dots, i_d)$ all permutations of $i_1, \dots, i_d$ seen as a multiset
    (i.e., there are $d!$ of them),
    and by $\mathcal{B}(\{{\dot{l}_{d_1}} \},(i_1,\dots, i_{d_1-1}))$ all bijections between the two sets.

    The entry of $\phi(\mathbf{M})$ in position indexed by $(\e_{i_1},\dots, \e_{i_{d_1-1}})$ and $\e_j$ is
    \begin{equation*}
        \begin{split}
            & \sum_{\pi \in S_{d_1}({i_1},\dots, {i_{d_1-1}}, {i_j})} \prod_{s=1}^{d_1-1} a_{s,\pi(i_s)}a_{d_1, \pi(i_{j})} 
            = \sum_{\pi \in S_{d_1}({i_1},\dots, {i_{d_1-1}}, {i_j})} \prod_{s=1}^{d_1-1} a_{\pi^{-1}(i_s),i_s}a_{\pi^{-1}(i_j), i_{j}} \\
            &=\sum_{\pi \in \mathcal{B}(\{{\dot{l}_{d_1}} \},(i_1,\dots, i_{d_1-1}))} \prod_{s=1}^{d_1-1} a_{\pi^{-1}(i_s),i_s}\sum_{l=1}^{d_1}a_{l, i_{j}} 
            =\sum_{l=1}^{d_1}\sum_{\pi \in \mathcal{B}(\{{\dot{l}_{d_1}} \},(i_1,\dots, i_{d_1-1}))} \prod_{s=1}^{d_1-1} a_{\pi^{-1}(i_s),i_s}a_{l, i_{j}}
        \end{split}
    \end{equation*}
    and so this part of the matrix factors as a matrix $Z$ whose entry in position indexed by $(\e_{i_1},\dots, \e_{i_{d_1-1}})$ as above and $l \in [d_1]$ is 
    \[\sum_{\pi \in \mathcal{B}(\{{\dot{l}_{d_1}} \},(i_1,\dots, i_{d_1-1}))} \prod_{s=1}^{d_1-1} a_{\pi^{-1}(i_s),i_s}\]
    and the matrix $W_1$.

    Similarly for an entry in position indexed by 
    $(\e_{i_1},\dots, \e_{i_{d_1-1}})$ and $k$ we get
    \begin{equation*}
        \sum_{\pi \in \mathcal{B}(\{{\dot{l}_{d_1}} \},(i_1,\dots, i_{d_1-1}))} \prod_{s=1}^{d_1-1} a_{\pi^{-1}(i_s),i_s}\sum_{l=1}^{d_1}b_{k,l} 
        =\sum_{l=1}^{d_1}\sum_{\pi \in \mathcal{B}(\{{\dot{l}_{d_1}} \},(i_1,\dots, i_{d_1-1}))} \prod_{s=1}^{d_1-1} a_{\pi^{-1}(i_s),i_s}b_{k,l}.
    \end{equation*}
    By concatenating the two parts of the matrix $\phi(\mathbf{M})$ we obtain  $\phi(\mathbf{M}) = Z[W_2^\top|W_1]$.
\end{proof}


\subsection{Proof of Proposition~\ref{prop:closed-form-binary-deep-network}}
\label{apx:closed-form-binary-deep-network}
\begin{proof}
    According to the recursive formula of $p^{(i)}$ and $q^{(i)}$ in Theorem~\ref{theorem:neural-network-closed-form-general}, we can rewrite the vector $p^{(i+1)}$ as
    \[
        \begin{bmatrix}
            p^{(i+1)}_1\\
            p^{(i+1)}_2
        \end{bmatrix}=
        \begin{bmatrix}
            w_{i+1,1,1}p^{(i)}_2+w_{i+1,1,2}p^{(i)}_1\\
            w_{i+1,2,1}p^{(i)}_2+w_{i+1,2,2}p^{(i)}_1
        \end{bmatrix}=W_{i+1}P_{12}\begin{bmatrix}
            p^{(i)}_1\\
            p^{(i)}_2
        \end{bmatrix}.
    \]
    Then, if we continue expanding, we will arrive at
    \begin{equation}
    \label{eq:binary-recursive-step}
        p^{(i+1)}=\begin{bmatrix}
            p^{(i+1)}_1\\
            p^{(i+1)}_2
        \end{bmatrix}=(W_{i+1}P_{12}\dots W_2P_{12}W_1)\xx.
    \end{equation}
    For the $q^{(i+1)}$ observe that 
    \[
        q^{(i+1)}=p^{(i)}_1p^{(i)}_2=(w_{i,1,1}p^{(i-1)}_2+w_{i,1,2}p^{(i-1)}_1)(w_{i,2,1}p^{(i-1)}_2+w_{i,2,2}p^{(i-1)}_1)
    \]
    where we can rewrite it as
    \[
        q^{(i+1)}=\begin{bmatrix}
            p^{(i-1)}_1 & p^{(i-1)}_2
        \end{bmatrix}^TW_{i}^TP_{12}W_i\begin{bmatrix}
            p^{(i-1)}_1\\ 
            p^{(i-1)}_2
        \end{bmatrix}
    \]
    where from~\eqref{eq:binary-recursive-step} the desired follows. Therfore, we can see that
    \[
        p^{(i)} = p_i\text{ and }q^{(i)}=q_i.
    \]
\end{proof}

\subsection{Proof of Lemma~\ref{lemma:binary-degrees}}
\label{apx:binary-degrees}
\begin{proof}
    If $L$ is even, then the degree $n(\dd)$ of the numerator is equal to
    \[
        (L/2-1)\cdot2 + 1 = L-1
    \]
    since $\deg(p_{i,L})=1$ and $\deg(q_k)=2$ for all $k=2,\dots,L$ with $\deg(q_1)=\deg(q_0)=0$.
    The degree $m(\dd)$ of the denominator is equal to
    \[
        (L/2)\cdot2 = L.
    \]
    In the same way, we can check then $L$ is odd, the degrees of $n(\dd)$ and $m(\dd)$ are equal to $L$ and $L-1$, respectively. From the other side, observe that 
    \[
        L+\delta(L)-1,\quad\text{ and }L-\delta(L)
    \]
    exactly gives $(L-1,L)$ if $L$ is even and $(L,L-1)$ if $L$ is odd where $\delta(L)=0$ if $L$ is even and $1$ if $L$ is odd. Therefore, $n(\dd)=L+\delta(L)-1$ and $m(\dd)=L-\delta(L)$.
\end{proof}

\subsection{Proof of Lemma~\ref{lemma:binary-params-dim}}
\label{apx:binary-params-dim}
\begin{proof}
    
    According to the formula of ambient dimension~\eqref{eq:dim-ambient-space-formula} and the degrees of the numerator and denominator from~\eqref{eq:coordinate-output-degree-of-BRNN}, we obtain
    \[
        M(L,d_L)=d_L\binom{2+n(\dd)-1}{1}+\binom{2+m(\dd)-1}{1}=d_L(1+n(\dd))+(1+m(\dd))=
    \]
    \[
        =d_L(L+\delta(L))+(L-\delta(L)+1).
    \]

    For a fixed $L>2$, we want to find an upper bound for the output dimension $d_L\geq1$ such that $N(L,d_L)\geq M(L,d_L)$ where $N(L,d_L)= 4(L-1) + 2d_L$. Direct computation shows
    \[
        4L - 4 + 2d_L \geq d_L(L + \delta(L)) + L - \delta(L) + 1,
    \]
    \[
        d_L \leq \frac{3L + \delta(L) - 5}{L + \delta(L) - 2} = 3 + \frac{1 - 2\delta(L)}{L + \delta(L) - 2}.
    \]
    Observe that $0<\frac{1 - 2\delta(L)}{L + \delta(L) - 2}<1$ for all even $L$ and $-1<\frac{1 - 2\delta(L)}{L + \delta(L) - 2}<0$ for all odd $L$. Therefore, $d_L\in\{1,2,3\}$ if $L$ is even and $d_L\in\{1,2\}$ if $L$ is odd.
\end{proof}

\subsection{Proof of Proposition~\ref{prop:binary-2..21-architecture}}
\label{apx:binary-2..21-architecture}
\begin{proof}
    Let $n(L)\coloneqq\deg P_{1,\ww}$ and $m(L)\coloneqq\deg Q_{\ww}$ be the degrees of the numerator and denominator, respectively, which depend on the number of layers $L$. Define $V_P\subset S^{n(L)}(\C^2)$ and $V_Q\subset S^{m(L)}(\C^2)$ to be the discriminant hypersurfaces (i.e., the vanishing locus of forms with a repeated linear factor), and set
    \[
        U \;\coloneqq\; \big(S^{n(L)}(\C^2)\setminus V_P\big)\times\big(S^{m(L)}(\C^2)\setminus V_Q\big).
    \]
    We claim that $U\subset \mathcal{M}_{\dd,\sigma}$, which implies
    \[
        \mathcal{V}_{\dd,\sigma}=S^{n(L)}(\C^2)\times S^{m(L)}(\C^2),
    \]
    since the neurovariety is the Zariski closure of the neuromanifold $\mathcal{M}_{\dd,\sigma}$. For a given $(P,Q)\in U$, our goal is to reconstruct the parameters $\ww=(W_1,W_2,\dots,W_L)$ such that $P_{1,\ww}=P$ and $Q_{\ww}=Q$. We will prove this by induction on the number of layers $L$. 
    
    \textbf{Base case $L=2$.}  
    Consider a neural network with architecture $\dd=(2,2,1)$. Its output numerator and denominator take the form
    \[
        P_{1,\ww}(\xx) \;=\; W_2 P_{12} W_1 \xx, 
        \qquad 
        Q_{\ww}(\xx) \;=\; \xx^T W_1^T P_{12} W_1 \xx.
    \]
    Fix $(P,Q)\in U$ with coefficients
    \[
        P(\xx) \;=\; 
        \begin{bmatrix} C_{1} & C_{2} \end{bmatrix}
        \begin{bmatrix} x \\ y \end{bmatrix}
        \;=\; C_P \xx,
    \]
    \[
        Q(\xx) \;=\; C_{11}x^2 + 2 C_{12}xy + C_{22}y^2 
        \;=\; 
        \begin{bmatrix} x & y \end{bmatrix}
        \begin{bmatrix} C_{11} & C_{12}\\ C_{12} & C_{22} \end{bmatrix}
        \begin{bmatrix} x \\ y \end{bmatrix}
        \;=\; \xx^TC_Q \xx.
    \]
    In other words, we must solve the system
    \[
        W_2 P_{12} W_1 \;=\; C_P,
        \qquad 
        W_1^T P_{12} W_1 \;=\; C_Q
    \]
    for the unknowns $W_1$ and $W_2$. Observe that once $W_1$ is reconstructed, $W_2$ follows immediately as
    \[
        W_2 \;=\; C_P W_1^{-1} P_{12}.
    \]
    Thus it suffices to determine $W_1$.  
    Without loss of generality, assume $C_{11}\neq 0$. Then the quadratic form $Q$ admits the factorization
    \[
        Q(x,y) \;=\; C_{11}\Big(x - \tfrac{-C_{12} - \sqrt{C_{12}^2 - C_{11}C_{22}}}{C_{11}}\, y\Big)
                        \Big(x - \tfrac{-C_{12} + \sqrt{C_{12}^2 - C_{11}C_{22}}}{C_{11}}\, y\Big).
    \]
    Hence one possible choice of $W_1$ is
    \[
        W_1 \;=\;
        \begin{bmatrix}
            C_{11} & C_{12} + \sqrt{C_{12}^2 - C_{11}C_{22}} \\
            1 & (C_{12} - \sqrt{C_{12}^2 - C_{11}C_{22}})/C_{11}
        \end{bmatrix}.
    \]
    Therefore, $W_1$ and subsequently $W_2$ can be reconstructed, completing the base case $L=2$.
    
    \textbf{Induction step.} 
    Now, let us proceed with the induction step. Assume that we can reconstruct the weights of any network with $L-1$ layers, and we want to show that we can reconstruct the network with $L$ layers. The idea is that for a given pair $(P,Q)$ corresponding to the architecture with $L$ layers, we construct polynomials $(P',Q')$ corresponding to the architecture with $L-1$ layers by first reconstructing the weights $W_1$.  
    
    Let us reconstruct $W_1$. Observe that the quadratic form $q_2$ corresponds to $W_1$, since
    \[
        q_2(\xx) \;=\; \xx^T W_1^T P_{12} W_1 \xx.
    \]
    If $L$ is even, then $q_2$ is a factor of $Q_{\ww}$, while if $L$ is odd, then $q_2$ is a factor of $P_{1,\ww}$.
    
    Without loss of generality, let $L$ be even. Over $\C$, any binary form splits into a product of linear factors, and since we removed the discriminant hypersurfaces, all linear factors of $P$ and $Q$ are pairwise non-proportional. Pick two distinct linear factors of $Q$, say $\tilde\ell_1,\tilde\ell_2$. Since they are not proportional, they determine an invertible matrix $W_1$ by taking its rows to be the coefficient vectors of $\tilde\ell_1$ and $\tilde\ell_2$.  
    
    Define
    \[
        P^{\prime}(\xx)\coloneqq P(W_1^{-1}\xx), 
        \qquad
        Q^{\prime}(\xx)\coloneqq \frac{Q(W_1^{-1}\xx)}{x_1x_2},
    \]
    where $\deg Q^{\prime}=m(L)-2=m(L-1)$ and $\deg P^{\prime}=n(L-1)$. Indeed, since $L$ is even, we have $n(L)=L-1$ and $m(L)=L$ by Lemma~\ref{lemma:binary-degrees}, while $n(L-1)=L-1$ and $m(L-1)=L-2$, exactly matching the degrees of $(P^{\prime},Q^{\prime})$ when $L-1$ is odd.  
    
    Thus, we have produced polynomials $P^{\prime}$ and $Q^{\prime}$ lying in the ambient space corresponding to a binary neural network with $L-1$ layers. By the induction hypothesis, we can reconstruct its weights and obtain $\ww^{\prime}=(W_2^{\prime},\dots,W_L^{\prime})$.  
    
    Finally, to recover the original weights of the network, observe that
    \[
        P(\xx) = P^{\prime}(W_1\xx) = P^{\prime}_{1,\ww^{\prime}}(W_1\xx)=P_{1,\ww}(\xx), 
        \qquad 
        Q(\xx) = q_2(\xx)\, Q^{\prime}(W_1\xx) = q_2(\xx)\, Q^{\prime}_{\ww^{\prime}}(W_1\xx)=Q_{\ww}(\xx).
    \]
    Hence the full parameter tuple is
    \[
        \ww = (W_1, W_2^{\prime}, \dots, W_L^{\prime}).
    \]
\end{proof}

\subsection{Proof of Proposition~\ref{prop:binary-not-dense-arch}}
\label{apx:binary-not-dense-arch}
\begin{proof}
    Note that the numerators $P_1,\ldots, P_k$ have the form~\eqref{eq:RNN-coordinates-formulas-binary}, so they share all the same factors except for one linear form. This poses a strong restriction on their coefficients giving rise to many polynomial constraints. For example, the resultant of any $P_i$ and $P_j$ is 0, for $i\neq j$, meaning that the neuromanifold is cut out by at least $\binom{k}2$ polynomial equations.
\end{proof}

\subsection{Proof of Lemma~\ref{lemma:dim-generic-fiber}}
\label{apx:dim-generic-fiber}
\begin{proof}
    Similarly to Lemma~\ref{lemma:fiber-symmetries-general}, if we do not cancel out all the entries of the diagonal matrices $D_i$ when we apply transformation of the parameters to the layer $(d_i,d_{i+1})$, then all output homogeneous polynomials will be rescaled by a product of all diagonal elements of $D_i$ when we go through all layers, which we denote by $\bar{\lambda}$, i.e.,
    \[
        (\bar\lambda P_1,\dots,\bar\lambda P_{d_L},\bar\lambda Q).
    \]
    But, to keep the combinatorial parameter map invariant under this transformation, we need to set $\bar\lambda=1$ which decreases the dimension of a generic fiber by $1$.
\end{proof}

\end{document}